\documentclass[12pt,a4paper]{article}
\usepackage{latexsym,amsfonts,amsmath,graphics}
\usepackage{epsfig}
\usepackage{enumitem}

\newtheorem{theorem}{Theorem}
\newtheorem{lemma}{Lemma}

\newtheorem{proposition}{Proposition}

\newenvironment{proof}{\begin{trivlist}
\item[\hskip\labelsep{\it Proof.}]}{$\hfill\Box$\end{trivlist}}

\newcommand{\norm}[1]{\left\Vert#1\right\Vert}

\newcommand{\bsa}{\boldsymbol{a}}
\newcommand{\bszeta}{\boldsymbol{\zeta}}

\newcommand{\bsalpha}{\boldsymbol{\alpha}}

\newcommand{\bsk}{\boldsymbol{k}}
\newcommand{\bsl}{\boldsymbol{l}}

\newcommand{\bsb}{\boldsymbol{b}}
\newcommand{\bsx}{\boldsymbol{x}}
\newcommand{\bsh}{\boldsymbol{h}}

\newcommand{\bsone}{\boldsymbol{1}}

\newcommand{\bsy}{\boldsymbol{y}}

\newcommand{\cG}{{\cal G}}

\newcommand{\e}{{\varepsilon}}

\newcommand{\uu}{\mathfrak{u}}
\newcommand{\de}{{\rm e}}

\newcommand{\icomp}{\mathtt{i}}
\newcommand{\bszero}{\boldsymbol{0}}
\newcommand{\rd}{\,\mathrm{d}}
\newcommand{\NN}{\mathbb{N}}
\newcommand{\ZZ}{\mathbb{Z}}

\newcommand{\nat}{\NN}

\newcommand{\qed} {\hfill \Box \vspace{0.5cm}}
\renewcommand{\pmod}[1]{\,(\bmod\,#1)}

\newcommand{\il}{\left<}
\newcommand{\ir}{\right>}
\def\qed{\hfill$\Box$}
\newcommand{\abs}[1]{\left\vert#1\right\vert}

\allowdisplaybreaks

\begin{document}

\title{Approximation of analytic
functions  in Korobov spaces}

\author{Josef Dick\thanks{J.~Dick is 
supported by an Australian Research Council Queen Elizabeth 2 Fellowship.}, 
Peter Kritzer\thanks{P.~Kritzer 
acknowledges the support of the Austrian Science Fund (FWF), 
Project P23389-N18.},
\\ Friedrich Pillichshammer, 
Henryk Wo\'zniakowski\thanks{H. Wo\'zniakowski is partially supported
by NSF.}}

\maketitle

\begin{abstract}
We study multivariate $L_2$-approximation for a weighted Korobov space
of analytic periodic functions for which
the Fourier coefficients decay exponentially fast.
The weights are defined, in particular, 
in terms of two sequences $\bsa=\{a_j\}$
and $\bsb=\{b_j\}$ of numbers no less than one.
Let $e^{L_2-\mathrm{app},\Lambda}(n,s)$ 
be the minimal worst-case error of all algorithms
that use $n$ information functionals from 
the class~$\Lambda$ in the $s$-variate case. We consider 
two classes $\Lambda$:
the class $\Lambda^{{\rm all}}$ 
consists of all linear functionals 
and the class $\Lambda^{{\rm std}}$ 
consists of only function evaluations.

We study (EXP) exponential convergence. This means that
$$
e^{L_2-\mathrm{app},\Lambda}(n,s)
\le C(s)\,q^{\,(n/C_1(s))^{p(s)}}\quad\mbox{for all}\quad n, s \in \NN 
$$
where $q\in(0,1)$, and $C,C_1,p:\NN \rightarrow (0,\infty)$.
If we can take $p(s)=p>0$ for all $s$  
then we speak of (UEXP) uniform exponential convergence.   
We also study EXP and UEXP with
(WT) weak, (PT) polynomial and (SPT) strong polynomial tractability. 
These concepts are defined as follows.
Let $n(\e,s)$ be the minimal $n$ for which 
$e^{L_2-\mathrm{app},\Lambda}(n,s)\le \e$.
Then WT holds iff
$\lim_{s+\log\,\e^{-1}\to\infty}(\log n(\e,s))/(s+\log\,\e^{-1})=0$,
PT holds iff there are 
$c,\tau_1,\tau_2$ such that
$n(\e,s)\le cs^{\tau_1}(1+\log\,\e^{-1})^{\tau_2}$
for all $s$ and $\e\in(0,1)$, and finally SPT holds iff the last estimate 
holds for $\tau_1=0$. 
The infimum of $\tau_2$ for which SPT holds is called
the exponent $\tau^*$ of SPT.   
We prove that the results are the same for both
classes $\Lambda$, and: 
\begin{itemize}
\item
EXP holds for any  $\bsa$, $\bsb$ and
$\omega$.
\item UEXP holds iff
$B:=\sum_{j=1}^\infty1/b_j<\infty$ 
and the largest $p$ 
is $1/B$.
\item
WT+EXP 
holds 
iff $\lim_ja_j=\infty$.
\item
WT+UEXP holds  iff $B<\infty$ and 
$\lim_ja_j=\infty$.
\item 
The notions of PT and 
SPT with EXP or UEXP are equivalent, 
and hold 
iff $B<\infty$ and 
$\alpha^*:=\liminf_{j\to\infty}(\log a_j)/j>0$. Then 
$$
\max(B,(\log3)/\alpha^*)\le
\tau^*\le B+(\log3)/\alpha^*,
$$ 
and $\tau^* = B$ for $\alpha^* =\infty$.
\end{itemize}
\end{abstract}

\section{Introduction}

We study approximation of $s$-variate functions defined 
on the unit cube $[0,1]^s$ with the worst-case error measured in the
$L_2$ norm. 
Multivariate approximation is a problem that has been studied 
in a vast number of papers from many different perspectives. 
We consider analytic periodic functions
belonging to a weighted Korobov 
space.  We present necessary and sufficient conditions 
on the decay of the Fourier coefficients 
under which we can achieve exponential and uniform exponential 
convergence with various notions of tractability.

We approximate functions by 
algorithms that use $n$ information evaluations.
We either allow information evaluations 
from the class $\Lambda^{\rm{all}}$ of
all continuous linear functionals or from
the class
$\Lambda^{\rm{std}}$ of standard information
which consists of only function evaluations. 

For large~$s$, it is important to study how the errors of algorithms
depend not only on $n$ but also on~$s$. 
The information complexity $n^{L_2-\mathrm{app},\Lambda}(\e,s)$ 
is the minimal number $n$ for which
there exists an
algorithm using $n$ information evaluations from the class 
$\Lambda\in\{\Lambda^{\rm{all}},\Lambda^{\rm{std}}\}$ with 
an error 
at most $\e$ in the $s$-variate case.
The information complexity 
is proportional to the minimal cost of computing an $\e$-approximation
since linear algorithms are optimal and their cost is proportional to 
$n^{L_2-\mathrm{app},\Lambda}(\e,s)$. 

We would like to control how $n^{L_2-\mathrm{app},\Lambda}(\e,s)$ 
depends on $\e^{-1}$ and $s$. 
In the standard study of tractability, 
see~\cite{NW08,NW10,NW12}, \emph{weak tractability} means that 
$n^{L_2-\mathrm{app},\Lambda}(\e,s)$ is \emph{not} exponentially
dependent  on $\e^{-1}$ and $s$. Furthermore, 
\emph{polynomial tractability} means that
$n^{L_2-\mathrm{app},\Lambda}(\e,s)$ is polynomially bounded 
by $C\,s^{\,q}\,\e^{-p}$ for some $C,q$ and $p$ independent of $\e\in(0,1)$
and $s\in\nat$. If $q=0$ then we have \emph{strong polynomial
tractability}.

Typically, $n^{L_2-\mathrm{app},\Lambda}(\e,s)$ is polynomially 
dependent on $\e^{-1}$ and $s$
for weighted classes of smooth functions.
The notion of weighted function classes means that
the successive variables and groups of variables are moderated by certain
weights. For sufficiently fast decaying weights,
the information complexity
depends at most polynomially on $s$,
and we obtain 
\emph{polynomial} tractability,
or even 
\emph{strong polynomial} tractability.

These notions of tractability are suitable for problems for which
smoothness of functions is
finite. This means that functions are differentiable only finitely many times.
Then the minimal errors of algorithms enjoy polynomial convergence
and are bounded by $C(s)\,n^{-\tau}$, for
some positive 
$C(s)$ which depends only on $s$ 
and some positive $\tau$ which depends on the smoothness
of functions. For many classes of such functions we know the largest $\tau$
which grows with increasing smoothness and decreasing weights.
Furthermore, 
weak tractability holds if $\log\,C(s)=o(s)$, whereas 
polynomial tractability holds if 
$C(s)$ is polynomially dependent on $s$, and
strong polynomial tractability holds if $C(s)$ is uniformly bounded
in $s$. 

It seems to us that the case of analytic or 
infinitely many times differentiable functions
is also of interest. For such classes of
functions we would like to replace polynomial convergence by 
exponential convergence, and study the same notions of tractability 
in terms of $(1+\log\,\e^{-1},s)$ instead of $(\e^{-1},s)$. 
More precisely, let
$e^{L_2-\mathrm{app},\Lambda}(n,s)$ 
be the minimal worst-case error
among all algorithms that use 
$n$ information evaluations from 
a permissible class~$\Lambda$ in the $s$-variate case.
By exponential convergence of the $n$th minimal approximation error
we mean that
$$
e^{L_2-\mathrm{app},\Lambda}(n,s)\le C(s)\, 
q^{\,(n/C_1(s))^{\,p(s)}}\ \ \ \ \mbox{for all}\ \ \ \ n,s\in\NN.
$$
Here, $q\in(0,1)$ is independent of $s$, 
whereas $C,C_1,$ and $p$ are allowed to be dependent on~$s$. 
We speak of uniform exponential convergence 
if $p$ can be replaced by a positive number 
independent of $s$. 
A priori it is not obvious what we should 
require about $C(s), C_1(s)$ and $p(s)$
although, clearly, the smaller $C(s)$ and $C_1(s)$ the better, 
and we would like to have $p(s)$ as large as possible. 
Obviously, if we do not care about the dependence on $s$
then the mere existence of $C(s), C_1 (s)$ and $p(s)$ is enough. 

The last bound on $e^{L_2-\mathrm{app},\Lambda}(n,s)$ yields
$$
n^{L_2-\mathrm{app},\Lambda}(\e,s) \le 
\left\lceil C_1(s) \left(\frac{\log C(s) +
    \log \e^{-1}}{\log q^{-1}}\right)^{1/p(s)}\right\rceil
\ \ \ \ \ \mbox{for all}\ \ \ s\in \NN\ \ \mbox{and}\ \ \e\in (0,1).
$$
Exponential convergence implies that asymptotically with
respect to $\e$ tending to zero, we need $\mathcal{O}(\log^{1/p(s)}
\e^{-1})$ information evaluations to compute 
an $\e$-approximation to functions from the Korobov space. 
(Throughout the paper $\log$ means the natural logarithm 
and $\log^{\,r} x$ means $[\log x]^r$.)

Tractability with exponential or uniform exponential 
convergence means that we would like to replace
$\e^{-1}$ by $1+\log\,\e^{-1}$ 
and guarantee the same properties 
on $n^{L_2-\mathrm{app},\Lambda}(\e,s)$ as for the standard case.
This means that (WT) weak tractability holds iff
$$
\lim_{s+\log\,\e^{-1} \rightarrow \infty}
\frac{\log\,n^{L_2-\mathrm{app},\Lambda}(\e,s)}{s+\log\,\e^{-1}}=0,
$$
whereas (PT) polynomial tractability holds iff there are non-negative
numbers $c,\tau_1,\tau_2$ such that
$$
n^{L_2-\mathrm{app},\Lambda}(\e,s)\le
c\,s^{\tau_1}(1+\log\,\e^{-1})^{\tau_2}\quad\mbox{for all}
\quad s\in \NN,\ \e\in(0,1).
$$
If $\tau_1=0$ in the last bound we speak of 
(SPT) strong polynomial tractability, and then~$\tau^*$ being the infimum 
of $\tau_2$ is called the exponent of SPT. 

For instance, uniform exponential convergence implies weak
tractability if
$$
C(s)=\exp\left(\exp\left(o(s)\right)\right)
\ \ \ \mbox{and}\ \ \ C_1(s)=\exp(o(s))
\ \ \ \ \
\mbox{as}\ \ \ \ \ s\to\infty.
$$
These conditions are rather weak since
$C(s)$ can be almost doubly exponential and $C_1(s)$ almost
exponential in $s$.

Furthermore, uniform exponential convergence 
implies polynomial tractability if for some non-negative $\eta_1$ and
$\eta_2$ we have 
$$
C(s)=\exp\left(\mathcal{O}(s^{\eta_1})\right)
\ \ \ \mbox{and}\ \ \ C_1(s)=\mathcal{O}(s^{\eta_2})
\ \ \ \ \
\mbox{as}\ \ \ \ \ s\to\infty.
$$
If $\eta_1=\eta_2=0$ then we have strong polynomial
tractability.

Uniform exponential convergence with weak, polynomial and strong polynomial
tract\-ability was studied in the papers~\cite{DLPW11}
and~\cite{KPW12} 
for multivariate integration in weighted Korobov spaces
with exponentially fast decaying Fourier coefficients. 
However, the notion of weak tractability was defined 
differently in a more demanding way, see Section~\ref{secint} 
for more details.
In the current paper, we deal with  multivariate approximation 
in the worst-case setting for the same class of functions.
We study exponential and
uniform exponential convergence and various notions of tractability
defined as above. 

We find it interesting that all results presented 
in this paper are exactly the same for both classes 
$\Lambda^{\rm all}$ 
and $\Lambda^{\rm std}$.
This is surprising since the class $\Lambda^{\rm std}$
is much smaller than the class $\Lambda^{\rm all}$.
This is very good news since 
usually in the computational practice we can only use 
function values, i.e., the class $\Lambda^{\rm std}$.
Furthermore, all our results are constructive
or semi-constructive\footnote{Semi-construction is only used for the class
$\Lambda^{\rm std}$ when we want to achieve WT with UEXP, 
see Section~\ref{semiconstruction}.}. 
That is, we provide
algorithms that use only function values and 
for which we achieve exponential and uniform exponential convergence
with WT, PT or SPT.  
The sample points used by these algorithms are from regular grids with
varying mesh-sizes for successive variables. 
Such grids were also successfully used for multivariate 
integration in the previous papers~\cite{DLPW11} and~\cite{KPW12}.

For the Korobov class of functions $f$ considered here, 
the decay of the Fourier coefficients $\widehat f(\bsh)$ is defined by two
sequences $\bsa=\{a_j\}$ and $\bsb=\{b_j\}$, 
and by a parameter $\omega\in(0,1)$. Here $\bsa$ and $\bsb$ are two sequences
of real numbers bounded below by 1, see Section~\ref{secKor} 
for further details. We assume that
$$
\sum_{\bsh\in \ZZ^s}|\widehat
f(\bsh)|^2\,\omega_{\bsh}^{-1} < \infty,
$$
where
$$
\omega_{\bsh}=\omega^{\,\sum_{j=1}^{s}a_j \abs{h_j}^{b_j}}
\qquad\mbox{for all}\qquad \bsh=(h_1,h_2,\dots,h_s)\in\ZZ^s.
$$
We study for which
$(\bsa,\bsb,\omega)$ we have exponential and uniform exponential
convergence without or with various notions of tractability. 
It turns out that 
$\omega$ only effects the factors in our
estimates. These factors go to infinity as $\omega$ tends to one.

We are going to show that  exponential convergence holds 
for any choice of $\bsa$ and $\bsb$, 
whereas 
uniform exponential convergence holds iff
$$
B:=\sum_{j=1}^\infty\frac1{b_j}<\infty,
$$
independently of $\bsa$.
Furthermore, the largest rate $p(s)$ for exponential convergence 
is $1/B(s)$, where
$$
B(s)=\sum_{j=1}^s \frac{1}{b_j},
$$
and for uniform exponential convergence the largest rate $p$ is $1/B$.

We prove that 
(WT+EXP) weak tractability with exponential
convergence holds iff
$$
\lim_{j\to\infty}a_j=\infty,
$$
and (WT+UEXP) weak tractability with uniform exponential convergence
holds iff 
$$
B<\infty\ \ \ \ \ \mbox{and}\ \ \ \ \ \lim_{j\to\infty}a_j=\infty.
$$
The notions of polynomial and
strong polynomial tractability with exponential or uniform exponential 
convergence are equivalent.
Furthermore, 
the strongest notion of tractability, namely strong 
polynomial tractability with uniform exponential convergence,
holds iff 
$$
B<\infty\quad\mbox{and}\quad 
\alpha^*=\liminf_{j\to\infty}\frac{\log\,a_j}j>0,
$$
and then the exponent $\tau^*$ of SPT satisfies
$$
\max\left(B,\frac{\log\,3}{\alpha^*}\right)\le
\tau^*\le B+\frac{\log\,3}{\alpha^*}.
$$

We comment on the assumption that $\alpha^*>0$. This means 
that the $a_j$ are exponentially large in $j$ for large $j$. 
Indeed, $\alpha^*>0$
implies that for any $\delta\in(0,\alpha^*)$ there is $j^*_\delta$
such that
\begin{equation}\label{newexpoaj}
a_j\ge \exp(\delta\, j)
\quad \mbox{for all}\quad j\ge j^*_\delta.
\end{equation}
Obviously, it may happen that $\alpha^*=\infty$. Then we know the
exponent of SPT exactly,
$$
\tau^*=B.
$$
Note that this happens if, for instance, $a_j\ge \exp(\alpha\, b_j)$
for large $j$ and  for some $\alpha>0$. Indeed, then  
$$
\alpha^*\ge \liminf_{j\to\infty}\,\frac{\alpha\, b_j}j =\infty,
$$
since $B<\infty$ implies that $\liminf_{j\to\infty}b_j/j=\infty$.

The rest of the paper is structured as follows. We give detailed 
information on the Korobov space in 
Section~\ref{secKor}, and on $L_2$-approximation 
and tractability in Section~\ref{secl2}. 
Our main results are summarized in Section~\ref{secmain}. 
The proofs for the class $\Lambda^{\rm{all}}$ 
are in Section~\ref{secproofall} using preliminary observations from 
Section~\ref{secprelall}. 
The proofs for the class $\Lambda^{\rm{std}}$ 
are in Section~\ref{secproofstd} using
preliminary observations from 
Section~\ref{secprelstd}. In Section~\ref{secint} we compare 
the approximation problem considered in this paper 
with the integration problem considered in \cite{KPW12}. 
Analyticity of functions from 
the Korobov space considered in this paper is
shown in Section~\ref{analyticfunctions}.

\section{The Korobov space $H(K_{s,\bsa,\bsb})$}\label{secKor}

The Korobov space $H(K_{s,\bsa,\bsb})$ 
discussed in this section is a Hilbert space with a reproducing kernel. 
For general information on reproducing kernel Hilbert 
spaces we refer to~\cite{Aron}.

Let $\bsa=\{a_j\}_{j \ge 1}$ and
$\bsb=\{b_j\}_{j \ge 1}$ be two sequences of real positive weights
such that
\begin{equation}\label{coefficients}
b_j\ge 1\ \ \ \ \ \mbox{and}\ \ \ \ \
a_j\ge1\ \ \ \ \ \ \mbox{for all}\ \ \
j=1,2,\dots\,.
\end{equation}
Throughout the paper we assume, without loss of generality, that 
$$
a_1 \le a_2 \le a_3 \le \ldots.
$$
Fix $\omega\in(0,1)$. Denote
\[
\omega_{\bsh}=\omega^{\sum_{j=1}^{s}a_j \abs{h_j}^{b_j}}
\qquad\mbox{for all}\qquad \bsh=(h_1,h_2,\dots,h_s)\in\ZZ^s.
\]
We consider a Korobov space of complex-valued one-periodic
functions defined on $[0,1]^s$ with a reproducing
kernel of the form
\begin{equation*}
K_{s,\bsa,\bsb}(\bsx,\bsy) =
\sum_{\bsh \in \ZZ^s} \omega_{\bsh}\,
\exp(2\pi
\icomp \bsh \cdot(\bsx-\bsy)) \ \ \ \mbox{for all}\ \ \
\bsx,\bsy\in[0,1]^s,
\end{equation*}
with the usual dot product
$$
\bsh\cdot(\bsx-\bsy)=\sum_{j=1}^sh_j(x_j-y_j),
$$ 
where
$h_j,x_j,y_j$ are the $j$th components of the vectors
$\bsh,\bsx,\bsy$, respectively, and $\icomp=\sqrt{-1}$.

The kernel $K_{s,\bsa,\bsb}$ is well
defined since
\begin{equation}\label{finitesumweights}
|K_{s,\bsa,\bsb}(\bsx,\bsy)|\le
K_{s,\bsa,\bsb}(\bsx,\bsx)=
\prod_{j=1}^s \left(1+2 \sum_{h=1}^{\infty} \omega^{a_j h^{b_j}}\right)<\infty.
\end{equation}
The last series is indeed finite since
$$
\sum_{h=1}^{\infty} \omega^{a_j h^{b_j}}\le
\sum_{h=1}^{\infty} \omega^{h}=\frac{\omega}{1-\omega}<\infty.
$$

The Korobov space with reproducing kernel
$K_{s,\bsa,\bsb}$ is a reproducing kernel Hilbert space 
and is denoted by $H(K_{s,\bsa,\bsb})$. 
We suppress the dependence on $\omega$ in the notation
since $\omega$ will be fixed throughout the paper and $\bsa$ and
$\bsb$ will be varied.

Clearly, functions from $H(K_{s,\bsa,\bsb})$ are infinitely many
times differentiable, see \cite{DLPW11}. 
They are also analytic
as shown in Section~\ref{analyticfunctions}.

For $f\in H(K_{s,\bsa,\bsb})$ we have
$$
f(\bsx)=\sum_{\bsh\in\ZZ^s} \widehat f(\bsh)\,\exp(2\pi \icomp
\bsh \cdot\bsx) \ \ \ \mbox{for all}\ \ \ \bsx\in [0,1]^s,
$$
where $\widehat{f}(\bsh) =
\int_{[0,1]^s} f(\bsx) \exp(-2 \pi \icomp \bsh \cdot \bsx) \rd \bsx$
is the $\bsh$th Fourier coefficient. 
The inner product of $f$ and $g$ from $H(K_{s,\bsa,\bsb})$ is given by
$$
\il f,g\ir_{H(K_{s,\bsa,\bsb})}=\sum_{\bsh\in \ZZ^s}\widehat f(\bsh)\,
\overline{\widehat g(\bsh)}\, \omega_{\bsh}^{-1}
$$ and the norm of $f$ from $H(K_{s,\bsa,\bsb})$ by 
$$
\|f\|_{H(K_{s,\bsa,\bsb})}=\left(\sum_{\bsh\in \ZZ^s}|\widehat
f(\bsh)|^2\omega_{\bsh}^{-1}\right)^{1/2}<\infty.
$$
Define the functions
\begin{equation}\label{basis}
e_{\bsh}(\bsx)=\exp(2\pi\icomp\,\bsh\cdot\bsx)\,
\omega_{\bsh}^{1/2}\ \ \ \ \
\mbox{for all}\ \ \ \ \ \bsx \in[0,1]^s.
 \end{equation}
Then $\{e_{\bsh}\}_{\bsh\in\ZZ^s}$ is a complete
orthonormal basis of the Korobov space $H(K_{s,\bsa,\bsb})$.

Integration of functions from $H(K_{s,\bsa,\bsb})$ was already
considered in \cite{KPW12} and, in the case $a_j=b_j=1$ for all $j \in
\NN$, also in \cite{DLPW11}. In this paper we consider
the problem of multivariate approximation in the $L_2$ norm
which we shortly call $L_2$-approximation.

\section{$L_2$-approximation}\label{secl2}

In this section we consider $L_2$-approximation of functions from
$H(K_{s,\bsa,\bsb})$.
This problem is defined as an approximation
of the embedding from the Korobov space $H(K_{s,\bsa,\bsb})$ to the space
$L_2([0,1]^s)$, i.e.,
$${\rm EMB}_s:H(K_{s,\bsa,\bsb}) \rightarrow L_2([0,1]^s)\ \ \ \ \
\mbox{given by}\ \ \ \ \ {\rm EMB}_s(f)=f.
$$

Without loss of generality, see, e.g., \cite{TWW88},
we approximate ${\rm EMB}_s$
by  linear algorithms~$A_{n,s}$ of the form
\begin{equation}\label{linalg}
  A_{n,s}(f) = \sum_{k=1}^{n}\alpha_k L_k(f)\ \ \ \
\mbox{for} \ \ \ \ \ f \in H(K_{s,\bsa,\bsb}),
\end{equation}
where each $\alpha_k$ is a function from $L_{2}([0,1]^{s})$ and
each $L_k$ is a continuous linear functional defined on $H(K_{s,\bsa,\bsb})$
from a permissible class $\Lambda$ of information. We consider two
classes:

\begin{itemize}
\item $\Lambda=\Lambda^{\mathrm{all}}$ , the class of all continuous
linear functionals defined on $H(K_{s,\bsa,\bsb})$.
Since  $H(K_{s,\bsa,\bsb})$ is a Hilbert space then for every $L_k\in
\Lambda^{\mathrm{all}}$ there exists a 
function $f_k$ from $H(K_{s,\bsa,\bsb})$ such that
$L_k(f)=\il f,f_k\ir_{H(K_{s,\bsa,\bsb})}$ for all $f\in
H(K_{s,\bsa,\bsb})$.
\item $\Lambda=\Lambda^{\mathrm{std}}$, the class of
standard information consisting only of function evaluations. 
That is, $L_k\in\Lambda^{\mathrm{std}}$ iff there exists
$\bsx_k\in[0,1]^{s}$ such that $L_k(f)=f(\bsx_k)$ for all $f\in
H(K_{s,\bsa,\bsb})$.
\end{itemize}

Since $H(K_{s,\bsa,\bsb})$ is a reproducing kernel Hilbert space,
function evaluations are continuous linear functionals and therefore
$\Lambda^{\mathrm{std}}\subseteq \Lambda^{\mathrm{all}}$. More
precisely,
$$
L_k(f)=f(\bsx_k)=\il
f,K_{s,\bsa,\bsb}(\cdot,\bsx_k)\ir_{H(K_{s,\bsa,\bsb})}
\ \ \ \mbox{and}\ \ \
\|L_k\|=\|K_{s,\bsa,\bsb}\|_{H(K_{s,\bsa,\bsb})}=
K^{1/2}_{s,\bsa,\bsb}(\bsx_k,\bsx_k).
$$

The \textit{worst-case error} of the algorithm $A_{n,s}$ is
defined as
\[
  e^{L_2-\mathrm{app}}(H(K_{s,\bsa,\bsb}),A_{n,s})
  :=  \sup_{f \in H(K_{s,\bsa,\bsb}) \atop \norm{f}_{H(K_{s,\bsa,\bsb})}\le 1}
  \norm{f-A_{n,s}(f)}_{L_{2}([0,1]^s)}.
\]
Let $e^{L_2-\mathrm{app},\Lambda}(n,s)$ be the $n$th minimal
worst-case error, 
$$
e^{L_2-\mathrm{app},\Lambda}(n,s) = \inf_{A_{n,s}}
e^{L_2-\mathrm{app}}(H(K_{s,\bsa,\bsb}),A_{n,s}),
$$ 
where the infimum is taken
over all linear algorithms $A_{n,s}$ using information
from the class $\Lambda$.
For $n=0$ we simply approximate $f$ by zero, and
the initial error is
$$
e^{L_2-\mathrm{app},\Lambda}(0,s) = \|{\rm EMB}_s\|= 
\sup_{f \in H(K_{s,\bsa,\bsb}) \atop \norm{f}_{H(K_{s,\bsa,\bsb})}\le
  1} 
\norm{f}_{L_{2}([0,1]^s)} = 1.
$$
This means that $L_2$-approximation is well normalized for all
$s\in\NN$.

\vskip 1pc

We study exponential convergence in this paper.
Suppose first that $s\in\NN$ is fixed.
Then  we hope that everyone would agree that exponential
convergence for $e^{L_2-\mathrm{app},\Lambda}(n,s)$ means
that there exist functions $q:\NN\to(0,1)$ and $p,C:\NN\to (0,\infty)$
 such that
$$
e^{L_2-\mathrm{app},\Lambda}(n,s)\le C(s)\, q(s)^{\,n^{\,p(s)}}
\ \ \ \ \ \mbox{for
all} \ \ \ \ \ n\in \NN.
$$

Obviously, the functions $q(\cdot)$ and $p(\cdot)$ are not uniquely defined.
For instance, we can take an arbitrary number $q\in(0,1)$,
define the function $C_1$ as
$$
C_1(s)=\left(\frac{\log\,q}{\log\,q(s)}\right)^{1/p(s)}
$$
and then
$$
C(s)\, q(s)^{\,n^{\,p(s)}}=C(s)\,q^{\,(n/C_1(s))^{p(s)}}.
$$
We prefer to work with the latter bound which was already considered
in~\cite{KPW12} for multivariate integration.

We say that we achieve  \emph{exponential convergence}
for $e^{L_2-\mathrm{app},\Lambda}(n,s)$ if
there exist a number $q\in(0,1)$ and 
functions $p,C,C_1:\NN\to (0,\infty)$ such that
\begin{equation}\label{exrate}
e^{L_2-\mathrm{app},\Lambda}(n,s)\le C(s)\, q^{\,(n/C_1(s))^{\,p(s)}}
\ \ \ \ \ \mbox{for
all} \ \ \ \ \ n\in \NN.
\end{equation}
If \eqref{exrate} holds we would like to find the largest possible
rate $p(s)$ of exponential convergence
defined as
\begin{equation}\label{exratemax}
p^*(s)=\sup\{\,p(s)\ :\ \ p(s)\ \ \mbox{satisfies \eqref{exrate}}\,\}.
\end{equation}

We say that we achieve \emph{uniform exponential convergence}
for $e^{L_2-\mathrm{app},\Lambda}(n,s)$ if the function $p$
in \eqref{exrate} can be taken as a constant function, i.e.,
$p(s)=p>0$ for all $s\in\NN$. Similarly, let
$$
p^*=\sup\{\,p\ :\ \ p(s)=p>0\ \ \mbox{satisfies \eqref{exrate} for all
$s\in\NN$}\,\}
$$
denote the largest rate of uniform exponential convergence.
\vskip 1pc

For $\varepsilon\in (0,1)$, $s\in \NN$, and
$\Lambda\in\{\Lambda^{\mathrm{all}},\Lambda^{\mathrm{std}}\}$,
the \emph{information complexity} is defined as
\[
  n^{L_2-\mathrm{app},\Lambda}(\varepsilon,s):=
  \min\left\{n\,:\,e^{L_2-\mathrm{app},\Lambda}(n,s)\le\varepsilon \right\}.
\]
Hence, $n^{L_2-\mathrm{app},\Lambda}(\varepsilon,s)$ is the minimal
number of information evaluations from $\Lambda$
which is required  to reduce the initial error
$e_{0,s}^{L_2-\mathrm{app}}$, which is one in our case, by a
factor of $\varepsilon \in (0,1)$. Clearly
$$
n^{L_2-\mathrm{app},\Lambda^{\mathrm{std}}}(\e,s)
\ge
n^{L_2-\mathrm{app},\Lambda^{\mathrm{all}}}(\e,s).
$$

We are ready to define tractability concepts similarly 
as in \cite{DLPW11} and \cite{KPW12}.
We stress again that these concepts correspond to the standard
concepts of tractability with  $\e^{-1}$ replaced by
$1+\log\,\e^{-1}$. We say that we have:
\begin{itemize}
\item \emph{Weak Tractability (WT)} if
$$
\lim_{s+\log\,\e^{-1}\to\infty}\frac{\log\
  n^{L_2-\mathrm{app},\Lambda}(\varepsilon,s)}{s+\log\,\e^{-1}}=0.
$$

Here we set $\log\,0=0$ by convention.
\item \emph{Polynomial Tractability (PT)} if there exist non-negative
  numbers $c,\tau_1,\tau_2$ such that
$$
n^{L_2-\mathrm{app},\Lambda}(\varepsilon,s)\le
c\,s^{\,\tau_1}\,(1+\log\,\e^{-1})^{\,\tau_2}\ \ \ \ \ \mbox{for all}\ \ \ \
s\in\NN, \ \e\in(0,1).
$$
\item
\emph{Strong Polynomial Tractability (SPT)} if there exist non-negative
  numbers $c$ and $\tau$  such that
$$
n^{L_2-\mathrm{app},\Lambda}(\varepsilon,s)\le
c\,(1+\log\,\e^{-1})^{\,\tau}\ \ \ \ \ \mbox{for all}\ \ \ \
s\in\NN, \ \e\in(0,1).
$$
The exponent $\tau^*$ of strong polynomial tractability is defined as
the infimum of $\tau$ for which  strong polynomial tractability holds.
\end{itemize}

A few comments of these notions are in order.
As in \cite{DLPW11}, we note that if \eqref{exrate} holds then
\begin{equation}\label{exrate2}
n^{L_2-\mathrm{app},\Lambda}(\e,s)
\le \left\lceil C_1(s) \left(\frac{\log C(s) +
    \log \e^{-1}}{\log q^{-1}}\right)^{1/p(s)}\right\rceil
\ \ \ \ \ \mbox{for all}\ \ \ s\in \NN\ \ \mbox{and}\ \ \e\in (0,1).
\end{equation}
Furthermore, if~\eqref{exrate2} holds then
$$
e^{L_2-\mathrm{app},\Lambda}(n+1,s)\le C(s)\, q^{\,(n/C_1(s))^{\,p(s)}}\ \
\ \ \ \mbox{for all}\ \ \ s,n\in \NN.
$$

This means that~\eqref{exrate} and~\eqref{exrate2} are practically
equivalent. Note that $1/p(s)$ determines the power of $\log\,\e^{-1}$
in the information complexity,
whereas $\log\,q^{-1}$ effects only the multiplier of $\log^{1/p(s)}\e^{-1}$.
{}From this point of view, $p(s)$ is more
important than $q$. That is why
we would like to have~\eqref{exrate} with the largest possible $p(s)$.
We shall see how to find such $p(s)$
for the parameters $(\bsa,\bsb,\omega)$ of the weighted Korobov space.

Exponential convergence implies that asymptotically, with
respect to $\e$ tending to zero, we need $\mathcal{O}(\log^{1/p(s)}
\e^{-1})$ information evaluations to compute an
$\e$-approximation to functions from the Korobov space. However, it is not
clear how long we have to wait to see this nice asymptotic
behavior especially for large $s$. This, of course, depends on
how $C(s),C_1(s)$ and $p(s)$ depend on $s$. This is
the subject of tractability which is extensively studied
in many papers. So far tractability has been studied 
in terms of $s$ and $\varepsilon^{-1}$. The current state of 
the art on tractability can be found in~\cite{NW08,NW10,NW12}. 
In this paper we follow the approach of \cite{DLPW11} and \cite{KPW12} 
and we study tractability in terms of $s$ and $1+\log
\varepsilon^{-1}$. 
In particular, weak tractability means that
we rule out the cases for which
$n^{L_2-\mathrm{app},\Lambda}(\e,s)$
depends exponentially on $s$ and $\log\,\e^{-1}$.

For instance, assume that~\eqref{exrate} holds.  Then 
uniform exponential convergence implies weak
tractability if
$$
C(s)=\exp\left(\exp\left(o(s)\right)\right)
\ \ \ \mbox{and}\ \ \ C_1(s)=\exp(o(s))
\ \ \ \ \
\mbox{as}\ \ \ \ \ s\to\infty.
$$
These conditions are rather weak since
$C(s)$ can be almost doubly exponential and $C_1(s)$ almost
exponential in $s$.
The definition of polynomial (and strong polynomial) tractability implies
that we have uniform exponential
convergence with $C(s)=\de$ (where $\de$ denotes $\exp (1)$), 
$q=1/\de$, $C_1(s)=c\,s^{\,\tau_1}$ and $p=1/\tau_2$.
For strong polynomial tractability $C_1(s)=c$ and $\tau^*\le1/p^*$.

If~\eqref{exrate2} holds then we have polynomial tractability  if
$p:=\inf_sp(s)>0$ and
there exist non-negative numbers $A,A_1$ and $\eta,\eta_1$ such that
$$
C(s)\le \exp\left(A s^{\eta}\right)\ \ \ \mbox{and}\ \ \
C_1(s)\le A_1\,s^{\eta_1}\ \ \ \ \
\mbox{for all}\ \ \ \ \
s\in\NN.
$$

The condition on $C(s)$ seems to be quite weak since even for
singly exponential $C(s)$ we have polynomial
tractability. Then $\tau_1=\eta_1+\eta/p$ and $\tau_2=1/p$.
Strong polynomial tractability holds
if $C(s)$ and $C_1(s)$ are uniformly bounded in $s$, and then
$\tau^*\le 1/p$.

\section{The main results}\label{secmain}

We first present the main results of this paper. We will be using the
following notational abbreviations
\begin{center}
EXP\qquad UEXP\qquad WT\qquad PT\qquad SPT\\
WT+EXP\qquad \ PT+EXP\qquad \ SPT+EXP\\
WT+UEXP\qquad  PT+UEXP\quad \ SPT+UEXP
\end{center}
to denote exponential and uniform exponential convergence,
weak, polynomial and strong polynomial tractability,
as well as
weak, polynomial and strong polynomial tractability with
exponential or uniform exponential convergence.
We want to find relations between these concepts as well as
necessary and sufficient conditions on $\bsa$ and $\bsb$ for which these
concepts hold. As we shall see, many of these concepts are equivalent.

\begin{theorem}\label{mainresult}
Consider $L_2$-approximation defined over the Korobov space with
kernel $K_{s,\bsa,\bsb}$ with arbitrary sequences $\bsa$ and 
$\bsb$ satisfying~\eqref{coefficients}.
The following results hold for both classes 
$\Lambda^{\rm{all}}$ and $\Lambda^{\rm{std}}$.
\begin{enumerate}[label=\arabic*,start=1] 
\item\label{allexp}
EXP holds for arbitrary $\bsa$ and $\bsb$ and
$$
p^{*}(s)=1/B(s) \ \ \ \ \ \mbox{with}\ \ \ \ \ B(s):=\sum_{j=1}^s\frac1{b_j}.
$$
This implies that
$$
\mbox{WT}\  \Leftrightarrow\   \mbox{WT+EXP},\ \ \ \ \
\mbox{PT}\  \Leftrightarrow\  \mbox{PT+EXP},\ \ \ \ \
\mbox{SPT}\ \Leftrightarrow\   \mbox{SPT+EXP}.
$$
\item\label{alluexp} 
UEXP holds iff $\bsa$ is an arbitrary sequence
and $\bsb$ such that
$$
B:=\sum_{j=1}^\infty\frac1{b_j}<\infty.
$$
If so then $p^*=1/B$ and
$$
\mbox{WT}\  \Leftrightarrow\   \mbox{WT+UEXP},\ \ \ \ \
\mbox{PT}\  \Leftrightarrow\  \mbox{PT+UEXP},\ \ \ \ \
\mbox{SPT}\ \Leftrightarrow\   \mbox{SPT+UEXP}.
$$

\item\label{allpt}
Polynomial (and, of course, strong polynomial)
tractability implies uniform exponential convergence,
$
\mbox{PT}\  \ \Rightarrow\ \  \mbox{UEXP},
$
i.e.,
$$
\mbox{PT}\  \Leftrightarrow\  \mbox{PT+UEXP},\ \ \ \ \
\mbox{SPT}\ \Leftrightarrow\   \mbox{SPT+UEXP}.
$$
\item\label{allwt}
We have
\begin{eqnarray*}
\mbox{WT}\ &\Leftrightarrow&\ \lim_{j\to\infty}a_j=\infty,\\
\mbox{WT+UEXP}\ &\Leftrightarrow&\  B<\infty\ \ \mbox{and}\ \
\lim_{j\to\infty}a_j=\infty.
\end{eqnarray*}
\item\label{allequiv}
The following notions are equivalent:
$$
\ \ \ \ \mbox{PT}\ \Leftrightarrow\ \mbox{PT+EXP}\ \Leftrightarrow\
\mbox{PT+UEXP}\ \Leftrightarrow\ \mbox{SPT}\ \Leftrightarrow\
\mbox{SPT+EXP}\ \Leftrightarrow\ \mbox{SPT+UEXP}.
$$
\item \label{allspt}
SPT+UEXP holds iff $b_j^{-1}$'s are summable and $a_j$'s are
exponentially large in $j$, i.e.,
$$
B:=\sum_{j=1}^\infty\frac1{b_j}<\infty\quad\mbox{and}\quad
\alpha^*:=\liminf_{j\to\infty}\frac{\log\,a_j}j>0.
$$
Then the exponent $\tau^*$ of SPT satisfies
$$
\max\left(B,\frac{\log\,3}{\alpha^*}\right)\le
\tau^*\le B+\frac{\log\,3}{\alpha^*}.
$$
In particular, if $\alpha^* =\infty$ then $\tau^* = B$.
\end{enumerate}
\end{theorem}

We comment on Theorem~\ref{mainresult}. 
We already expressed our surprise in the introduction
that the results are the same for both classes $\Lambda^{\rm std}$
and $\Lambda^{\rm all}$, although the class
$\Lambda^{\rm std}$  is much smaller than the class 
$\Lambda^{\rm all}$. However, the proofs for both classes are
different. We also stress that the results 
are constructive. The corresponding algorithms 
can be found in Section~\ref{secprelall} for the class 
$\Lambda^{\rm all}$ and in Section~\ref{secproofstd} for the class
$\Lambda^{\rm std}$.

Point \ref{allexp} tells us 
that we always have exponential convergence and the best rate is 
$p^*(s)=1/B(s)$. Note that $p^*(s)$ decays with $s$, and
if $B(s)$ goes to infinity then the rate decays to zero.
The smallest rate is for $b_j=1$ for all $j\ge 1$, for which $p^*(s)=1/s$.
Clearly, all tractability notions with or without exponential
convergence are trivially equivalent.

Point \ref{alluexp} addresses uniform exponential convergence which holds iff $b_j^{-1}$'s
are summable, i.e., when $B<\infty$.
Then the best rate of uniform exponential convergence
is $p^*=1/B$.  Obviously, for large $B$ this rate is poor.
We stress that uniform exponential convergence
holds independently of $\bsa$.
Similarly as before, as long as $B<\infty$, tractability notions with
or without uniform exponential convergence are trivially equivalent.

Point \ref{allpt} states that (strong) polynomial tractability implies
uniform exponential convergence, i.e., $B<\infty$. This means that the notion of
polynomial tractability is stronger than the notion of uniform
convergence.

Point \ref{allwt} addresses weak tractability which holds iff
$a_j$'s tend to infinity.
We stress that this holds independently of $\bsb$ and independently
of the rate of
convergence of~$\bsa$ to infinity.
We have weak tractability with uniform
convergence if additionally $B<\infty$.
Hence for $\lim_j a_j=\infty$ and $B=\infty$, weak tractability holds
without uniform exponential convergence. 

Point \ref{allequiv} states that, in particular,
the notions of polynomial tractability
and strong polynomial tractability with uniform exponential convergence
are equivalent. 
 
Point \ref{allspt} presents necessary and sufficient
conditions on strong polynomial tractability
with uniform exponential convergence. We must assume that $B<\infty$ and
$\alpha^*>0$. The last condition means that  
$a_j$'s are exponentially large in $j$  
for large~$j$.
We only know bounds of the exponent  $\tau^*$ 
of strong polynomial tractability. 
Note that for large $B$ or small $\alpha^*$ 
the exponent~$\tau^*$ is large. On the other hand, $\tau^*$ is not large 
if $B$ is not large and $\alpha^*$ is not small. 
We stress that $B$ can be sufficiently small if all $b_j$ are
sufficiently large, whereas $\alpha^*$ can be sufficiently large if
$a_j$ are large enough. In fact, we may even have $\alpha^*=\infty$.
This holds if $a_j$ goes to infinity faster than $C^j$ for any $C>1$. 
We already noticed in the introduction that this holds, for example, if
$a_j \ge \exp(\delta\, b_j)$ for large $j$ and for  some
$\delta > 0$. For $\alpha^*=\infty$ we know the exponent of SPT exactly,
$$
\tau^*=B.
$$

\section{Preliminaries for the class $\Lambda^{\mathrm{all}}$}\label{secprelall}

The information complexity is known for the class
$\Lambda^{\mathrm{all}}$, see, e.g., \cite[Chapter~4,
Section~5.8]{TWW88}). It depends on the eigenpairs of the operator
$$
W_s={\rm EMB}_s^*\,{\rm EMB}_s: \ H(K_{s,\bsa,\bsb})\to
H(K_{s,\bsa,\bsb}),
$$
which in our case is given by
$$
W_sf=\sum_{\bsh\in\ZZ^s}\omega_{\bsh}\il
f,e_{\bsh}\ir_{H(K_{s,\bsa,\bsb})}e_{\bsh}
$$
with $e_{\bsh}$ given by~\eqref{basis}. 
Hence, the eigenpairs of $W_s$ are $(\omega_{\bsh},e_{\bsh})$ since
$$
W_se_{\bsh}=\omega_{\bsh}e_{\bsh}=
\omega^{\,\sum_{j=1}^sa_j|h_j|^{b_j}}\, e_{\bsh}\ \ \ \ \
\mbox{for all}\ \ \ \ \ \bsh\in \ZZ^s.
$$
It is known that the information complexity is the number of the
eigenvalues $\omega_{\bsh}$ of the operator
$W_s$ which are greater than $\e^2$.
More precisely, for a real $M$  define the set
\begin{eqnarray}\label{eqAsMweighted}
\mathcal{A}(s,M)
& := & \left\{\bsh\in\ZZ^s\ : \ \omega_{\bsh}^{-1}< M\right\}\nonumber \\
& = & \left\{\bsh\in\ZZ^s\ : \ \omega^{-\sum_{j=1}^{s}a_j
\abs{h_j}^{b_j}}< M\right\}.
\end{eqnarray}
Then
\begin{equation}\label{opt_std}
n^{L_2-\mathrm{app},\Lambda^{\mathrm{all}}}(\e,s)
=\abs{\mathcal{A}(s,\varepsilon^{-2})}.
\end{equation}
Furthermore, the optimal algorithm in the class
$\Lambda^{\mathrm{all}}$ is the truncated Fourier series
\[
  A_{n,s}^{(\mathrm{opt})}(f)(\bsx)
  :=  \sum_{\bsh\in\mathcal{A}(s,\varepsilon^{-2})}
\il f,e_{\bsh}\ir_{H(K_{s,\bsa,\bsb})}e_{\bsh}=
  \sum_{\bsh\in\mathcal{A}(s,\varepsilon^{-2})}
  \widehat{f}(\bsh) \exp(2 \pi \icomp \bsh \cdot\bsx),
\]
where $n = \abs{\mathcal{A}(s,\varepsilon^{-2})}$, which ensures
that the worst-case error satisfies
$$
e^{L_2-\mathrm{app}}
(H(K_{s,\bsa,\bsb}),A_{n,s}^{(\mathrm{opt})})\le\varepsilon.
$$

For the proof of Theorem \ref{mainresult}
and also for the further considerations in this
paper we need a few properties of the set $\mathcal{A}(s,M)$
and its cardinality.
Clearly,  $\mathcal{A}(s,M)=\emptyset$ for all $M\le1$. For
$\e\in(0,1)$, let
$$
x=x(\e):=\frac{\log\,\e^{-2}}{\log\,\omega^{-1}}>0,
$$
and
$$
n(x,s):=
\left|\left\{\bsh\in \ZZ^s\,:\ \sum_{j=1}^sa_j|h_j|^{b_j}<x\right\}\right|.
$$
Then
$$
n^{L_2-\mathrm{app},\Lambda^{\mathrm{all}}}(\e,s)=
|\mathcal{A}(s,\e^{-2})|=n(x,s).
$$
We have $n(x,s)=1$ for all $x\in(0,a_1]$ and
\begin{eqnarray*}
n(x,1)&=&2\left\lceil (x/a_1)^{1/b_1}\right\rceil -1,\\
n(x,s)&=&n(x,s-1)+2\sum_{h=1}^{\left\lceil
    (x/a_s)^{1/b_s}\right\rceil-1}n(x-a_sh^{b_s},s-1).
\end{eqnarray*}

Clearly, $n(y,s)\ge n(x,s)\ge n(x,s-1)$ for all $y\ge x>0$ and
$s\ge2$. Note that for $x\le a_s$, the last sum in $n(x,s)$ is zero
and $n(x,s)=n(x,s-1)$. For $x>a_1$, define
$$
j(x)=\sup\{\,j \in\NN \ :  \ x>a_j\,\}.
$$
For $\lim_ja_j<\infty$ we have $j(x)=\infty$ for large $x$. For
$\lim_ja_j=\infty$, we can replace the supremum in $j(x)$ by the
maximum, and $j(x)$ is finite for all $x$. However, $j(x)$ tends to
infinity with $x$.

If $j(x)$ is finite then
$$
n(x,s)=n(x,j(x))\ \ \ \ \ \mbox{for all} \ \ \ \ \ s\ge j(x),
$$
and therefore, if $j(x)<\infty$ then
$$
\lim_{s\to\infty}\frac{\log\,n(x,s)}{s+x}=0.
$$
We now prove the following lemma.
\begin{lemma}\label{lemAsMweighted}\ \
\begin{itemize}
\item
For $x>a_1+a_2+\cdots+a_s$ we have
$$
n(x,s)\ge 3^s.
$$
\item
For $x>a_1$ and for arbitrary $\alpha_j\in[0,1]$ we have
\begin{eqnarray*}
n(x,s)&\ge&
\prod_{j=1}^{\min(s,j(x))}\left(2\left\lceil
\left(\frac{x}{a_j}\ (1-\alpha_j)\prod_{k=j+1}^s
\alpha_k\right)^{1/b_j}
\right\rceil-1\right),\\
n(x,s)&\le&
\prod_{j=1}^{\min(s,j(x))}\left(2\left\lceil
\left(\frac{x}{a_j}\right)^{1/b_j}\right\rceil-1\right),
\end{eqnarray*}
where the empty product is defined to be $1$.
\item
For $x>a_1$ we have
$$
\prod_{j=1}^s\left(2\left\lceil
\left(\frac{x}{a_j\,s}\right)^{1/b_j}
\right\rceil-1\right)\le n(x,s)\le
\prod_{j=1}^{\min(s,j(x))}\left(2\left\lceil
\left(\frac{x}{a_j}\right)^{1/b_j}\right\rceil-1\right).
$$
\end{itemize}
\end{lemma}

\begin{proof}
To prove the first point, let
$A_s=\{\,\bsh\in \ZZ^s\ :\ \ h_j\in\{-1,0,1\}\,\}$.
For $\bsh\in A_s$ we have
$$
\sum_{j=1}^sa_j|h_j|^{b_j}\le \sum_{j=1}^sa_j<x.
$$
Hence $3^s=|A_s|\le n(x,s)$, as claimed.

We turn to the second point.
It is easier to prove the upper bound on $n(x,s)$.
{}From the recurrence relation
on $n(x,s)$ we have
$$
n(x,s)\le n(x,s-1)+2\left(
\left\lceil\left(
      \frac{x}{a_s}\right)^{1/b_s}\right\rceil -1\right)\ n(x,s-1)=
\left(2\left\lceil\left(
      \frac{x}{a_s}\right)^{1/b_s}\right\rceil -1\right)\ n(x,s-1).
$$
This yields
$$
n(x,s)\le
\prod_{j=1}^s\left(
2\left\lceil
\left(\frac{x}{a_j}\right)^{1/b_j}\right\rceil-1\right).
$$
If $j > j(x)$, i.e., $x\le a_j$, then the factor
$$
2\left\lceil
\left(\frac{x}{a_j}\right)^{1/b_j}\right\rceil-1=2\cdot1-1=1.
$$
Hence, we can restrict $j$ in the last product to $\min(s,j(x))$
and obtain the desired upper bound on $n(x,s)$.

We turn to the lower bound on $n(x,s)$. Note that $x-a_sh^{b_s}>
\alpha_sx$ for all $h\in\NN$ with
$$
h\le
\left\lceil\left(\frac{x(1-\alpha_s)}{a_s}
\right)^{1/b_s}\right\rceil -1.
$$
Hence
\begin{eqnarray*}
n(x,s)&\ge& n(\alpha_sx,s-1)+
2n(\alpha_sx,s-1)
\left(\left\lceil\left(\frac{x(1-\alpha_s)}{a_s}
\right)^{1/b_s}\right\rceil -1\right)\\
&=&
\left(
2\left\lceil\left(\frac{x(1-\alpha_s)}{a_s}
\right)^{1/b_s}\right\rceil -1\right)\ n(\alpha_sx,s-1).
\end{eqnarray*}
We now apply induction on $s$. For $s=1$ we have
$$
n(x,1)= 2\left\lceil(x/a_1)^{1/b_1}\right\rceil -1\ge 
2\left\lceil((1-\alpha_1)x/a_1)^{1/b_1}\right\rceil -1,
$$
as claimed. Then
\begin{eqnarray*}
n(x,s)&\ge&\left(
2\left\lceil\left(\frac{x(1-\alpha_s)}{a_s}
\right)^{1/b_s}\right\rceil -1\right)\ \prod_{j=1}^{s-1}
\left(2\left\lceil
\left(\frac{\alpha_s x}{a_j}\ (1-\alpha_j)
\prod_{k=j+1}^{s-1}\alpha_k\right)^{1/b_j}
\right\rceil-1\right)\\
&\ge&
\prod_{j=1}^s
\left(2\left\lceil
\left(\frac{x}{a_j}\
  (1-\alpha_j)\prod_{k=j+1}^s\alpha_k\right)^{1/b_j}\right\rceil-1\right)\\
&=&
\prod_{j=1}^{\min(s,j(x))}
\left(2\left\lceil
\left(\frac{x}{a_j}\
  (1-\alpha_j)\prod_{k=j+1}^s\alpha_k\right)^{1/b_j}\right\rceil-1\right),
\end{eqnarray*}
as claimed. This completes the proof of the second point.

To prove the third point, it is enough to take $\alpha_j=(j-1)/j$.
Then for $j=1,2,\dots,s$ we have
$$
(1-\alpha_j)\,\prod_{k=j+1}^s\alpha_k=\frac{\prod_{k=j+1}^s(k-1)}
{j\,\prod_{k=j+1}^sk}=\frac1s,
$$
as claimed. This completes the proof of Lemma~\ref{lemAsMweighted}.
\hfill \qed

\section{The proof of Theorem~\ref{mainresult} 
for $\Lambda^{\rm{all}}$}\label{secproofall}

We are ready to prove Theorem~\ref{mainresult} for the class 
$\Lambda^{\rm{all}}$. 

\subsection{The proof of Point~\ref{allexp}}
{}From the second and third points of
Lemma~\ref{lemAsMweighted} with a fixed $s$ we have
$$
n(x,s)=\Theta(x^{B(s)})\ \ \ \ \ \mbox{as}\ \ \ \ \ x\to\infty.
$$
Therefore there are functions $c_1,c_2:\NN\to (0,\infty)$ such that
$$
c_1(s)\,\log^{B(s)}\,\e^{-1}\le
n^{L_2-\mathrm{app},\Lambda^{\mathrm{all}}}(\e,s)\le
c_2(s)\,\log^{B(s)}\,\e^{-1}
$$
for $\e$ tending to zero.
This implies exponential convergence since
$$
e^{L_2-\mathrm{app},\Lambda^{\mathrm{all}}}(n,s)\le q^{(n/c_2(s))^{1/B(s)}}
\ \ \ \ \ \mbox{with}
\ \ \ \ \ q=\exp(-1).
$$
Hence, $p^*(s)\ge 1/B(s)$. On the other hand, if we have exponential
convergence~\eqref{exrate} then
$$
n^{L_2-\mathrm{app},\Lambda^{\mathrm{all}}}(\e,s)=\mathcal{O}
\left(\log^{1/p(s)}\,\e^{-1}\right)
$$
and $1/p(s)\ge B(s)$, or equivalently, $p(s)\le1/B(s)$.
Hence, $p^*(s)=1/B(s)$,  
as claimed in Point~\ref{allexp}. The rest in this point is clear.

\subsection{The proof of Point~\ref{alluexp}}
Assume now that we have uniform exponential convergence.
Then $e^{L_2-\mathrm{app},\Lambda^{\mathrm{all}}}(n,s)\le
C(s)\,q^{\,(n/C_1(s))^p}$ implies for a fixed $s$ that
$$
n^{L_2-\mathrm{app},\Lambda^{\mathrm{all}}}(\e,s)=\mathcal{O}
(\log^{1/p}\,\e^{-1})\ \ \ \ \ \mbox{as}\ \ \ \ \ \e\to0.
$$
Then $B(s)\le 1/p$ for all $s$. Therefore $B\le 1/p<\infty$ and
$p^*\le 1/B$. On the other hand, if $B<\infty$ then we can set
$p(s)=1/B$ and obtain uniform exponential convergence.
Hence, $p^*\ge 1/B$, and therefore $p^*=1/B$, as claimed.
The rest of Point~\ref{alluexp} is clear.

\subsection{The proof of Point~\ref{allpt}}
PT means that
$$
n^{L_2-\mathrm{app},\Lambda^{\mathrm{all}}}(\varepsilon,s)\le
c\,s^{\,\tau_1}\,(1+\log\,\e^{-1})^{\,\tau_2}.
$$
This implies that
$$
e^{L_2-\mathrm{app},\Lambda^{\mathrm{all}}}(n)\le
\de^{1-(n/c\,s^{\tau_1})^{1/\tau_2}}.
$$
Hence, UEXP
holds with $p=1/\tau_2$. This also yields
the equivalence between various notions of tractability with or
without uniform exponential convergence.

\subsection{The proof of Point~\ref{allwt}}
We first prove that WT implies $\lim_ja_j=\infty$.
We use the first part of
Lemma~\ref{lemAsMweighted}.
For $\delta>0$, take
$x=(1+\delta)(a_1+\cdots+a_s)$, or equivalently
$$
\log\,\e^{-1}=\frac{x}{2}\log \omega^{-1}=
\frac{\log\,\omega^{-1}}2\,(1+\delta)(a_1+a_2+\cdots+a_s).
$$
Then
$$
z_s:=\frac{\log\,n(x,s)}{s+\log\,\e^{-1}}\ge
\frac{s\,\log\,3}{s+\tfrac12\,x\,\log\,\omega^{-1}}
=\frac{\log\,3}{1+\tfrac12\,(1+\delta)y_s\,\log\,\omega^{-1}},
$$
where
$$
y_s=\frac{a_1+a_2+\cdots+a_s}{s}.
$$
WT implies that $\lim_sz_s=0$. This can hold only if
$\lim_sy_s=\infty$ which implies that $\lim_ja_j=\infty$, as
claimed.

Next, we need to prove that
$\lim_ja_j=\infty$ implies WT. The eigenvalues of $W_s$ are
$\omega_{\bsh}$ for all $\bsh\in\ZZ^s$. Let the ordered eigenvalues of
$W_s$ be $\lambda_{s,n}$ for $n\in\nat$ with 
$\lambda_{s,1} \ge \lambda_{s,2}\ge \lambda_{s,3}\ge \ldots$.  Obviously 
$\{\lambda_{s,n}\}_{n\in\nat}=\{\omega_{\bsh}\}_{\bsh\in\ZZ^s}$.
Therefore for any $\eta\in(0,1)$ we have
$$
n\lambda_{s,n}^{\eta}\le\sum_{j=1}^\infty\lambda_{s,j}^{\eta}=
\sum_{\bsh\in\ZZ^s}\omega_{\bsh}^{\eta}=\prod_{j=1}^s\left(1+
2\sum_{h=1}^\infty\omega^{\,\eta\,a_j\,h^{b_j}}\right).
$$
Note that
$$
\sum_{h=1}^\infty\omega^{\,\eta\,a_j\,h^{b_j}}\le 
\sum_{h=1}^\infty\omega^{\,\eta\,a_j\,h}=\frac{\omega^{\eta\,a_j}}
{1-\omega^{\eta\,a_j}}.
$$ 
This proves that 
\begin{equation}\label{bdlambdasn}
\lambda_{s,n}\le\frac{
\prod_{j=1}^s
\left(1+2\,\omega^{\eta\,a_j}/(1-\omega^{\eta\,a_j})
\right)^{1/\eta}}{n^{1/\eta}}.
\end{equation}
Since 
$n^{L_2-\mathrm{app},\Lambda^{\rm all}}(\varepsilon,s)
=\min\{n:\
\lambda_{s,n+1}<\e^2\}$ we conclude that
$$
n^{L_2-\mathrm{app},\Lambda^{\rm all}}(\varepsilon,s)\le 
\frac{\prod_{j=1}^s
\left(1+2\,\omega^{\eta\,a_j}/(1-\omega^{\eta\,a_j})\right)}{\e^{2\eta}}.
$$
Using $\log(1+x)\le x$ for $x\ge0$, this yields
$$
\log\,n^{L_2-\mathrm{app},\Lambda^{\rm all}}(\varepsilon,s)\le 
2\,\eta\,\log\,\e^{-1}\ +\ 2\,\sum_{j=1}^sc_j,
$$ 
where
$$
c_j=\frac{\omega^{\eta\,a_j}}
{1-\omega^{\eta\,a_j}}.
$$
Note that $\lim_ja_j=\infty$ implies that $\lim_jc_j=0$, and 
$\lim_s\sum_{j=1}^sc_j/s=0$. Therefore
$$
\limsup_{s+\log\,\e^{-1}\to\infty}
\frac{\log\,n^{L_2-\mathrm{app},\Lambda^{\rm all}}(\varepsilon,s)}
{s+\log\,\e^{-1}}
\le 2\,\eta.
$$
Since $\eta$ can be arbitrarily small this proves that
$$
\lim_{s+\log\,\e^{-1}\to\infty}
\frac{\log\,n^{L_2-\mathrm{app},\Lambda^{\rm all}}(\varepsilon,s)}
{s+\log\,\e^{-1}}=0.
$$
Hence, WT holds for the class $\Lambda^{\rm{all}}$, as
claimed. 
The rest in this point follows from the previous results.
This completes the proof of Point~\ref{allwt}.

\subsection{The proof of Points~\ref{allequiv} and ~\ref{allspt}}

For Point~\ref{allequiv}, it is enough to prove that PT implies SPT+UEXP.
This will be done by showing that PT implies that $B<\infty$ and
$\alpha^*>0$. Then we show that
$B<\infty$ and $\alpha^*>0$ imply SPT+UEXP
and obtain bounds on the exponent of~SPT.

We know that PT implies UEXP and that 
UEXP implies that $B<\infty$. 
{}From the lower bound of Lemma~\ref{lemAsMweighted}
with  $x=(1+\delta)(a_1+\cdots +a_s)$ and from PT we have
$$
3^s\le n(x,s)\le
C\,s^{\tau_1}\,\left(1+ \frac{1+\delta}{2} (\log\,\omega^{-1})
(a_1+\cdots+a_s)\right)^{\tau_2}\ \ \ \ \mbox{for all}\ \ \ \ s\in\NN.
$$
Since $a_1 \le a_2 \le a_3 \le \ldots$, this yields
$$
s\,a_s\ge a_1+\cdots+a_s\ge \
\frac{2}{(1+\delta)\log\,\omega^{-1}} \
\left[\left(\frac{3^s}{C\,s^{\tau_1}}\right)^{1/\tau_2}-1\right]
\ \ \ \ \ \mbox{for all}\ \ \ \ \ s\in\NN.
$$
Hence,
$$
\alpha^*=\liminf_{s\to\infty}\frac{\log\,a_s}s\ge
\frac{\log\,3}{\tau_2}>0,
$$
as needed. This also shows that $\tau_2\ge (\log3)/\alpha^*$. 
Since this holds for all $\tau_2$ for which we have SPT,
we conclude that the exponent $\tau^*$ of SPT also satisfies
$\tau^*\ge (\log3)/\alpha^*$. Clearly, $\tau^*$ cannot be smaller than
the reciprocal of the exponent $p^*$ of UEXP. Hence, $\tau^*\ge B$. 
This completes this part of the proof as well as the proof of
lower bounds on the exponent of SPT.

Assume now that $B<\infty$ and $\alpha^* \in (0, \infty]$.
{}From~\eqref{newexpoaj} with $\delta\in(0,\alpha^*)$ we have 
$$
a_j\ge \exp(\delta j) \quad\mbox{for all}\quad j\ge
j^*_{\delta}.
$$
Then
$$
j(x)\le \max\left(j^*_{\delta},\frac{\log\,x}{\delta}\right).
$$
For $x>a_1$, the upper bound on $n(x,s)$ from
Lemma~\ref{lemAsMweighted} yields
\begin{eqnarray}\label{bdnxs}
n(x,s)&\le&
\prod_{j=1}^{\min(s,j(x))}
\left(1+2\left(\frac{x}{a_j}\right)^{1/b_j}\right)\nonumber\\
&\le&
\left[\prod_{j=1}^{\min(s,j(x))}\left(\frac{x}{a_j}\right)^{1/b_j}\right] \
3^{\min(s,j(x))}\nonumber \\
&\le&
x^{\,B}\,\max\left(3^{j^*_\delta},x^{\,(\log\,3)/
\delta}\right)\nonumber \\
&\le&3^{j^*_\delta}\,x^{B+(\log\,3)/\delta}.
\end{eqnarray}
Hence, SPT+UEXP holds, as claimed. Furthermore, since $\delta$ can be
arbitrarily close to $\alpha^*$, we conclude that the exponent of SPT satisfies
$$
\tau^*\le B \,+\,\frac{\log\,3}{\alpha^*},
$$
where for $\alpha^* = \infty$ we have $\frac{\log \,3}{\alpha^*} = 0$.
This completes the proof of Point~\ref{allequiv} and 
of Point~\ref{allspt}. The proof of 
the whole theorem for the class $\Lambda^{\rm{all}}$ is now completed.
\end{proof}

\section{Preliminaries for the class $\Lambda^{\mathrm{std}}$}\label{secprelstd}

We state some preliminary observations which will be needed to prove
Theorem~\ref{mainresult} for the class $\Lambda^{\rm{std}}$.
Based on the definition of the set $\mathcal{A}(s,M)$
in~\eqref{eqAsMweighted} for $M>1$, we will study 
approximating $f \in H(K_{s,\bsa,\bsb})$ by algorithms of the
form  
\begin{equation}\label{eqdefAnsM}
A_{n,s,M}(f)(\bsx)= \sum_{\bsh
\in \mathcal{A}(s,M)} \left(\frac{1}{n}\sum_{k=1}^{n} f(\bsx_k)
\exp(-2 \pi \icomp \bsh \cdot \bsx_k) \right) \exp(2 \pi \icomp
\bsh \cdot \bsx),
\end{equation}
where $\bsx\in[0,1]^s$.
Note that $A_{n,s,M}$ is a linear algorithm as in
\eqref{linalg} with 
$$\alpha_k(\bsx)
=\frac{1}{n}  \sum_{\bsh \in
\mathcal{A}(s,M)} \exp(2 \pi \icomp \bsh \cdot (\bsx-\bsx_k))$$
and with $L_k(f)=f(\bsx_k)$
for deterministically chosen sample
points $\bsx_k \in [0,1)^s$ for $1 \le k \le n$.
Hence, $L_k\in\Lambda^{\mathrm{std}}$. 
The choice of $M$ and $\bsx_k$ will be given later.

We first study upper bounds on the worst-case error of $A_{n,s,M}$. 
The following analysis is similar to that in \cite{KSW06}. We have
\begin{eqnarray*}
(f-A_{n,s,M}(f))(\bsx) &=&  
\sum_{\bsh \not \in  \mathcal{A}(s,M)} 
\widehat{f}(\bsh) \exp(2 \pi \icomp \bsh \cdot \bsx)\\
&&+ \sum_{\bsh \in
\mathcal{A}(s,M)}\left(\widehat{f}(\bsh)-\frac{1}{n}
\sum_{k=1}^{n}f(\bsx_k) \exp(-2 \pi \icomp \bsh \cdot
\bsx_k)\right) \exp(2 \pi \icomp \bsh \cdot \bsx).
\end{eqnarray*}
Using Parseval's identity we obtain
\begin{eqnarray}\label{approx_err}
\lefteqn{\|f-A_{n,s,M}(f)\|_{L_2([0,1]^s)}^2}\nonumber\\
& = & \sum_{\bsh \not \in  \mathcal{A}(s,M)} |\widehat{f}(\bsh)|^2
+ \sum_{\bsh \in \mathcal{A}(s,M)}
\left|\widehat{f}(\bsh)-\frac{1}{n}
\sum_{k=1}^n f(\bsx_k) \exp(-2 \pi \icomp \bsh \cdot \bsx_k)\right|^2\nonumber\\
& = & \sum_{\bsh \not \in  \mathcal{A}(s,M)} |\widehat{f}(\bsh)|^2 \nonumber\\
&& + \sum_{\bsh \in \mathcal{A}(s,M)}\left|\int_{[0,1]^s}
f(\bsx) \exp(-2 \pi \icomp \bsh \cdot \bsx) \rd
\bsx-\frac{1}{n}\sum_{k=1}^{n} f(\bsx_k) \exp(-2 \pi \icomp \bsh
\cdot \bsx_k)\right|^2.
\end{eqnarray}
We have 
\begin{equation}\label{approx_err1}
\sum_{\bsh \not \in  \mathcal{A}(s,M)}
|\widehat{f}(\bsh)|^2 = \sum_{\bsh \not \in  \mathcal{A}(s,M)}
|\widehat{f}(\bsh)|^2 \omega_{\bsh}\,\omega_{\bsh}^{-1} \le
\frac{1}{M} \|f\|_{H(K_{s,\bsa,\bsb})}^2.
\end{equation}

For the second term in \eqref{approx_err}, 
we make a specific choice for the points
$\bsx_1,\ldots,\bsx_n$ used in the algorithm $A_{n,s,M}$. 
Namely, we take $\bsx_j$'s from
a regular grid
with different mesh-sizes for successive variables.
Such regular grids have already been studied in~\cite{DLPW11, KPW12}.
We now 
recall their definition. For $s \in \NN$, a regular grid with 
mesh-sizes $m_1,\ldots,m_s
\in \NN$ is defined as the point set 
\begin{equation*}
\cG_{n,s}
=\left\{(k_1/m_1,\ldots, k_s/m_s)\, : \ \  
k_j=0,1,\ldots,m_j -1 \mbox{\ \  for all\  } j=1,2,\ldots ,s\right\},
\end{equation*}
where $n=\prod_{j=1}^sm_j$ is the cardinality of $\cG_{n,s}$.
By $\cG_{n,s}^{\bot}$ we denote the
dual of $\cG_{n,s}$, i.e., 
$$\cG_{n,s}^{\bot}=\{\bsh \in \ZZ^s \, : \, h_j \equiv 0 \pmod{m_j} 
\ \mbox{ for all } \ j=1,2,\ldots,s\}.
$$

We will make use of the following result whose easy proof is omitted.
\begin{lemma}\label{grile}
Let $\cG_{n,s}=\{\bsx_1,\ldots,\bsx_n\}$ be defined as above. 
{}For any $f\in H(K_{s,\bsa,\bsb})$ we have $$\left|\int_{[0,1]^s} f(\bsx)\rd
\bsx - \frac{1}{n}\sum_{k=1}^n f(\bsx_k)\right|=\left|\sum_{\bsh
\in \cG_{n,s}^{\bot} \setminus \{\bszero\}}
\widehat{f}(\bsh)\right|.$$
\end{lemma}
\vskip 1pc
For $\bsh \in \ZZ^s$ define $f_{\bsh}(\bsx):=f(\bsx) \exp(-2 \pi
\icomp \bsh \cdot \bsx)$. Note that with $f$ also $f_{\bsh}$
belongs to $H(K_{s,\bsa,\bsb})$ and that
$\widehat{f_{\bsh}}(\bsk)=\widehat{f}(\bsh+\bsk)$. 
{}From Lemma~\ref{grile} we obtain
\begin{eqnarray*}
\left|\int_{[0,1]^s} f_{\bsh}(\bsx)
\rd\bsx-\frac{1}{n}\sum_{k=1}^{n} f_{\bsh}(\bsx_k) \right|^2 & = &
\left|\sum_{\bsl \in \cG_{n,s}^{\bot} 
\setminus \{\bszero\}} \widehat{f_{\bsh}}(\bsl)\right|^2
= \left|\sum_{\bsl \in \cG_{n,s}^{\bot} \setminus \{\bszero\}} 
\widehat{f}(\bsl+\bsh)\right|^2\\
& \le & \left(\sum_{\bsl \in \cG_{n,s}^{\bot} \setminus \{\bszero\}}
\abs{\widehat{f}(\bsl+\bsh)}^2
\omega_{\bsh+\bsl}^{-1}\right)\left(\sum_{\bsl \in \cG_{n,s}^{\bot} 
\setminus \{\bszero\}}\omega_{\bsh+\bsl}\right)\\
& \le & \|f\|_{H(K_{s,\bsa,\bsb})}^2\left(\sum_{\bsl \in
    \cG_{n,s}^{\bot} 
\setminus
\{\bszero\}}\omega_{\bsh+\bsl}\right).
\end{eqnarray*}
Therefore, and using \eqref{approx_err} and \eqref{approx_err1} for any $f\in H(K_{s,\bsa,\bsb})$ with
$\|f\|_{H(K_{s,\bsa,\bsb})} \le 1$, we obtain
\begin{eqnarray}\label{gen_approx_bd}
\norm{f-A_{n,s,M}(f)}_{L_2([0,1]^s)}^2 & \le & \frac{1}{M} +
\sum_{\bsh\in\mathcal{A}(s,M)}\ \sum_{\bsl \in \cG_{n,s}^{\bot}
\setminus \{\bszero\}}\omega_{\bsh+\bsl}.
\end{eqnarray}

It is easy to see that
\[\abs{\ell}^b\le 2^b\left(\abs{h+\ell}^b + \abs{h}^b\right)\]
for any $h,\ell\in\ZZ$ and any $b\in\NN$. 
For $\bsh\in\mathcal{A}(s,M)$ this implies 
\begin{equation}\label{bds_omegakl}
\omega_{\bsh+\bsl}=\omega^{\sum_{j=1}^s a_j|h_j+\ell_j|^{b_j}}\le
\omega^{\,\sum_{j=1}^s 2^{-b_j} a_j\abs{\ell_j}^{b_j}}
  \omega^{-\sum_{j=1}^s a_j\abs{h_j}^{b_j}}\le 
\omega^{\,\sum_{j=1}^s 2^{-b_j} a_j\abs{\ell_j}^{b_j}} M.
\end{equation}

Using \eqref{gen_approx_bd}, \eqref{bds_omegakl} and
Lemma~\ref{lemAsMweighted} 
with $x=(\log M )/(\log \omega^{-1})$, 
we obtain for any $f \in H(K_{s,\bsa,\bsb})$ with
$\norm{f}_{H(K_{s,\bsa,\bsb})} \le 1$,
\begin{eqnarray*}
\norm{f-A_{n,s,M}(f)}_{L_2([0,1]^s)}^2 & \le &
 \frac{1}{M}+
M\abs{\mathcal{A}(s,M)}
\sum_{\bsl \in \cG_{n,s}^{\bot} \setminus
\{\bszero\}}
  \omega^{\,\sum_{j=1}^s 2^{-b_j} a_j\abs{\ell_j}^{b_j}}\\
&\le&\frac{1}{M}+ M \left(\prod_{j=1}^s\left(1+2\left(\frac{\log
M}{a_j \log \omega^{-1}}\right)^{1/b_j}\right)\right)
F_n,
\end{eqnarray*}
where
\[F_n := \sum_{\bsl \in \cG_{n,s}^{\bot} \setminus \{\bszero\}}
\omega^{\,\sum_{j=1}^s 2^{-b_j} a_j\abs{\ell_j}^{b_j}}.\]
This means that
\begin{equation}\label{eqboundMapp}
[e^{L_2-{\rm app}}(H(K_{s,\bsa,\bsb}),
A_{n,s,M})]^2\le \frac{1}{M}+ M \left(\prod_{j=1}^s\left(1+2\left(\frac{\log
M}{a_j \log \omega^{-1}}\right)^{1/b_j}\right)\right)
F_n.
\end{equation}
Furthermore,
\begin{eqnarray*}
 \prod_{j=1}^s\left(1+2\left(\frac{\log
M}{a_j \log \omega^{-1}}\right)^{1/b_j}\right)&\le&
2^s \prod_{j=1}^s\left(1+\left(\frac{\log
M}{\log \omega^{-1}}\right)^{1/b_j}\right)\\
&\le& 2^s \prod_{j=1}^s\left(1+\log^{-1/b_j}
\omega^{-1}\right)\prod_{j=1}^s\left(1+\log^{1/b_j}
M\right).
\end{eqnarray*}
Since $M$ is assumed to be at least 1, we can bound  $1+\log^{1/b_j}
M\le 2M^{1/b_j}$, and obtain
\[\prod_{j=1}^s\left(1+2\left(\frac{\log
M}{a_j \log \omega^{-1}}\right)^{1/b_j}\right)\le 4^s M^{B(s)}
\prod_{j=1}^s
\left(1+\log^{-1/b_j} \omega^{-1}\right),\]
where, as in the previous sections, $B(s):=\sum_{j=1}^s b_j^{-1}$. 
Plugging this into \eqref{eqboundMapp}, we obtain
\begin{equation}\label{eqboundMapp2}
[e^{L_2-{\rm app}}(H(K_{s,\bsa,\bsb}),A_{n,s,M})]^2\le
\frac{1}{M}+ 
M^{B(s)+1} D(s,\omega,\bsb) F_n,
\end{equation}
where
\[D(s,\omega,\bsb):=4^s \prod_{j=1}^s\left(1+\log^{-1/b_j}\omega^{-1}\right).\]

\section{The proof of Theorem~\ref{mainresult} for 
$\Lambda^{\rm{std}}$}\label{secproofstd}

We now present the proofs for the successive points of
Theorem~\ref{mainresult} for the class~$\Lambda^{\mathrm{std}}$.

\subsection{The proof of Point~\ref{allexp}} 

The following proposition will be helpful.
\begin{proposition}\label{th_upper_expo_conv}
For $s\in \NN$ and $\e\in(0,1)$ define
$$
m=\max_{j=1,2,\dots,s}\
\left\lceil \left(
\frac{4^{b_j}}{a_j}\,\frac{\log\left(1+\frac{2s}{\log(1+\eta^2)}\right)}
{\log\,\omega^{-1}}\right)^{B(s)}\,\right\rceil,
$$
where
\[\eta=\left(\frac{\e^2}{2D(s,\omega,\bsb)^{\frac{1}{B(s)+2}}}
\right)^{\frac{B(s)+2}{2}}.\]
Let
$\cG_{n,s}^{\ast}$ be a regular grid with mesh-sizes $m_1,m_2,\ldots,m_s$
given by
$$
m_j:=\left\lfloor m^{1/(B(s) \cdot b_j)}\right\rfloor\ \ \ \ \
\mbox{for}\ \ \  j=1,2,\ldots,s\ \ \
\mbox{and}\ \ \ n=\prod_{j=1}^sm_j.
$$
Then for $M=2/\e^2$ we have 
$$
e^{L_2-{\rm app}}(H(K_{s,\bsa,\bsb}),A_{n,s,M})\le\e,\ \ \ \ 
\mbox{and}\ \ \ \
n=\mathcal{O}\left(\log^{\,B(s)}\left(1+\e^{-1}\right)\right)
$$
with the factor in the $\mathcal{O}$ notation independent
of $\e^{-1}$ but dependent on $s$.
\end{proposition}
\begin{proof} 
We can write
\[
F_n=\sum_{\bsl \in \cG_{n,s}^{\bot} \setminus \{\bszero\}}
\omega^{\,\sum_{j=1}^s 2^{-b_j} a_j\abs{\ell_j}^{b_j}}=
-1+\prod_{j=1}^s\left(1+2\sum_{h=1}^\infty 
\omega^{a_j 2^{-b_j}(m_j h)^{b_j}}\right).
\]
Since $\lfloor x\rfloor\ge x/2$ for all $x\ge1$, we have
$$
|m_jh_j|^{b_j}\ge (|h_j|/2)^{b_j}\,m^{1/B(s)}
\qquad\mbox{for all}\qquad j=1,2,\dots,s.
$$ 
Hence,
$$
F_n \le -1+\prod_{j=1}^s
\left(1+2 \sum_{h=1}^{\infty} \omega^{m^{1/B(s)} a_j4^{-b_j}\,h^{b_j}}\right).
$$
Since $b_j\ge1$ we further estimate
$$
\sum_{h=1}^{\infty} \omega^{m^{1/B(s)} a_j4^{-b_j}\,h^{b_j}}\le
\sum_{h=1}^{\infty} \omega^{m^{1/B(s)} a_j4^{-b_j}\,h}=
\frac{\omega^{m^{1/B(s)}a_j4^{-b_j}}}{1-\omega^{m^{1/B(s)}a_j4^{-b_j}}}.
$$
{}From the definition of $m$ we have
$$
\frac{\omega^{m^{1/B(s)}a_j4^{-b_j}}}{1-\omega^{m^{1/B(s)}a_j4^{-b_j}}}
\le \frac{\log(1+\eta^2)}{2s}\qquad\mbox{for all}\qquad
j=1,2,\dots,s.
$$
This proves
\begin{equation}\label{eqFnbound}
F_n\le
-1+\left(1+\frac{\log(1+\eta^2)}{s}\right)^s
\le 
-1+\exp(\log(1+\eta^2))=\eta^2.
\end{equation}
Now, plugging this into \eqref{eqboundMapp2}, we obtain
\begin{equation}\label{eqboundMapp3}
[e^{L_2-{\rm app}}(H(K_{s,\bsa,\bsb}),A_{n,s,M})]^2\le
\frac{1}{M}+ 
M^{B(s)+1} D(s,\omega,\bsb) \eta^2.
\end{equation}
Note that
\[
\frac{1}{D(s,\omega,\bsb)^{\frac{1}{B(s)+2}} \eta^{\frac{2}{B(s)+2}}}=
\frac{2}{\e^2}\ge 1.
\]
Hence we are allowed to choose
\[M=\frac{1}{D(s,\omega,\bsb)^{\frac{1}{B(s)+2}} \eta^{\frac{2}{B(s)+2}}},\]
which yields, inserting into \eqref{eqboundMapp3},
\[[e^{L_2-{\rm app}}(H(K_{s,\bsa,\bsb}),A_{n,s,M})]^2\le 
2D(s,\omega,\bsb)^{\frac{1}{B(s)+2}} \eta^{\frac{2}{B(s)+2}}=\e^2,\]
as claimed.

It remains to verify that $n$ is of the order 
stated in the proposition. Note that 
$$
n=\prod_{j=1}^s
m_j=\prod_{j=1}^s 
\left\lfloor
m^{1/(B(s) \cdot b_j)} \right\rfloor \le m^{\frac{1}{B(s)}\sum_{j=1}^s
1/b_j} = m.
$$
However, as pointed out in \cite{KPW12},
\[m=\mathcal{O}\left(\log^{B(s)}\left (1 +\eta^{-1}\right)\right),\]
as $\eta$ tends to zero. {}From this, it is easy to see that we indeed have
\[m=\mathcal{O}\left(\log^{B(s)}\left (1 +\e^{-1}\right)\right),\]
which concludes the proof of Proposition~\ref{th_upper_expo_conv}.
\end{proof}
\vskip 1pc
To show Point~\ref{allexp} for the class $\Lambda^{\rm{std}}$,
we conclude from  Proposition~\ref{th_upper_expo_conv} that
\[n^{L_2-\mathrm{app},\Lambda^{\rm{std}}}(\varepsilon,s)=\mathcal{O}
\left(\log^{B(s)}\left (1 +\e^{-1}\right)\right).\]
This implies that we indeed have exponential convergence 
for $\Lambda^{\rm{std}}$ for all $\bsa$ and $\bsb$, 
with $p(s)=1/B(s)$, and thus
$p^* (s)\ge 1/B(s)$. On the other hand, note that 
obviously $e^{L_2-\mathrm{app},\Lambda^{\rm{std}}}(n,s)\ge 
e^{L_2-\mathrm{app},\Lambda^{\rm{all}}}(n,s)$, 
hence the rate of exponential convergence for $\Lambda^{\rm{std}}$ 
cannot be larger than for $\Lambda^{\rm{all}}$ which is $1/B(s)$. 
Thus, also for the class $\Lambda^{\rm{std}}$ we have
$p^*(s)= 1/B(s)$.  The rest of Point~\ref{allexp} is clear.

\subsection{The proof of Point~\ref{alluexp}}

We turn to Point~\ref{alluexp} for the class $\Lambda^{\rm{std}}$.
Suppose first that $\bsa$ is an arbitrary sequence
and that $\bsb$ is such that 
\[B=\sum_{j=1}^\infty \frac{1}{b_j}<\infty.\] 
Then we can replace $B(s)$ by $B$ in
Proposition~\ref{th_upper_expo_conv}, 
and we obtain
\[n^{L_2-\mathrm{app},\Lambda^{\rm{std}}}(\varepsilon,s)=\mathcal{O}
\left(\log^{B}\left (1 +\e^{-1}\right)\right),\] 
hence uniform exponential convergence with $p^* \ge 1/B$ holds. 
On the other hand, if we have uniform exponential convergence for 
$\Lambda^{\rm{std}}$, this
implies uniform exponential convergence for $\Lambda^{\rm{all}}$,
which in turn implies that $B<\infty$ and that $p^*  \le 1/B$. 
The rest of Point~\ref{alluexp} follows immediately.

\subsection{The proof of Point~\ref{allpt}}

The proof of Point~\ref{allpt} follows 
by the same arguments as for $\Lambda^{\rm{all}}$.

\subsection{The proof of Point~\ref{allwt}}\label{semiconstruction}

We now prove the first part of Point~\ref{allwt} for the class
$\Lambda^{\rm{std}}$. 
Assume that WT holds for the class $\Lambda^{\rm{std}}$.
Then WT also holds for the class $\Lambda^{\rm{all}}$ and this implies
that $\lim_ja_j=\infty$, as claimed. 

Assume now that $\lim_ja_j=\infty$. 
We use \cite[Theorem~26.18]{NW12}.
In particular, this theorem states that 
if the ordered eigenvalues $\lambda_{s,n}$'s of
$W_s$ satisfy 
\begin{equation}\label{2618}
\lambda_{s,n}\le \frac{M^{\,2}_{s,\tau}}{n^{2\tau}} \ \ \ \ \ \ \  
\mbox{for all} \ \ \ \ \ n\in\nat,
\end{equation}
for some positive $M_{s,\tau}$ and $\tau>\tfrac12$ then 
there is a semi-constructive algorithm\footnote{By semi-constructive
we mean that this algorithm can be constructed after 
a few random selections of sample points, more can be found in~\cite{NW12}.} 
such that
\begin{equation}\label{2619}
e^{L_2-\mathrm{app},\Lambda^{\rm std}}(n+2,s)\le
\frac{M_{s,\tau}\,C(\tau)}{n^{\tau(2\tau/(2\tau+1))}}\ \ \ \ \ \ \   
\mbox{for all} \ \ \ \ \ n\in\nat
\end{equation}
where $C(\tau)$ is given explicitly in \cite[Theorem~26.18]{NW12}.
However, the form of $C(\tau)$ is not important for our consideration.  

For $\eta\in(0,1)$, let $\tau=1/(2\eta)>\tfrac12$. 
We stress that $\tau$ can be arbitrarily large if we take sufficiently
small $\eta$. We already showed
in the proof for the class $\Lambda^{\rm all}$, 
see \eqref{bdlambdasn}, that we can take
$$
M_{s,\tau}=\prod_{j=1}^s\left(1+2
c_j\right)^{\tau}<\infty\quad\mbox{with}\quad 
c_j=\frac{\omega^{\,a_j/(2\tau)}}
{1-\omega^{\,a_j/(2\tau)}}.
$$
Furthermore, we know that $\lim_ja_j=\infty$ implies that
$\lim_s\sum_{j=1}^sc_j/s=0$. 

{}From \eqref{2619} we obtain 
$$
n^{L_2-\mathrm{app},\Lambda^{\mathrm{std}}}(\varepsilon,s)\le
3+\left(M_{s,\tau}\,C(\tau)\right)^{(1+1/(2\tau))/\tau}\,
\e^{-(1+1/(2\tau))/\tau}.
$$
This yields that
$$
\limsup_{s+\log\,\e^{-1}\to\infty}
\frac{\log\,n^{L_2-\mathrm{app},\Lambda^{\rm std}}
(\varepsilon,s)}{s+\log\,\e^{-1}}
\le 
\left(1+\frac1{2\tau}\right)\,\frac1{\tau}\,\left(1+
\limsup_{s\to\infty}
\frac{\log\,M_{s,\tau}}{s}\right).
$$
Since $
(\log M_{s,\tau})/s\le 
2\tau\,\sum_{j=1}^sc_j/s$ tends to zero as $s\rightarrow\infty$, we have
$$
\limsup_{s+\log\,\e^{-1}\to\infty}
\frac{\log\,n^{L_2-\mathrm{app},\Lambda^{\rm std}}
(\varepsilon,s)}{s+\log\,\e^{-1}}
\le 
\left(1+\frac1{2\tau}\right)\,\frac1{\tau}.
$$
Since $\tau$ can be arbitrarily large this proves that
$$
\lim_{s+\log\,\e^{-1}\to\infty}
\frac{\log\,n^{L_2-\mathrm{app},\Lambda^{\rm std}}
(\varepsilon,s)}{s+\log\,\e^{-1}}=0.
$$
This means that WT holds for the class $\Lambda^{\rm{std}}$, as
claimed. 

We turn to the second part of Point~\ref{allwt} for the class
$\Lambda^{\rm{std}}$. This point easily follows from the already proved
facts that WT holds iff $\lim_ja_j =\infty$ and UEXP holds iff $B<\infty$. 

\subsection{The proof of Point~\ref{allequiv}}

Suppose that PT holds for the class $\Lambda^{\rm{std}}$. 
Then PT holds for the class $\Lambda^{\rm{all}}$. 
By Point~\ref{allequiv} for the class $\Lambda^{\rm{all}}$, 
which has already been proved, this 
implies SPT+UEXP for the class $\Lambda^{\rm{all}}$ 
which in turn implies that $B<\infty$ and 
$\alpha^*>0$ by Point~\ref{allspt} for 
the class $\Lambda^{\rm{all}}$. 
This implies SPT+UEXP for the class $\Lambda^{\rm{std}}$ as will 
be shown in the subsequent Section~\ref{secstdspt}. 
The rest of this point is clear.

\subsection{The proof of Point~\ref{allspt}}\label{secstdspt}

The necessity of the conditions for SPT+UEXP on $\bsb$ and $\bsa$ 
stated in Point~\ref{allspt} for the class~$\Lambda^{\rm{std}}$  
follows from the same conditions for the class $\Lambda^{\rm{all}}$
and the fact that the information complexity for $\Lambda^{\rm{std}}$ cannot
be smaller than for $\Lambda^{\rm{all}}$.

To prove the sufficiency of the conditions for SPT+UEXP on $\bsb$ and
$\bsa$ stated  in Point~\ref{allspt} we 
analyze the algorithm $A_{n,s,M}$ given by~\eqref{eqdefAnsM},
where the sample points $\bsx_k$ are  from the regular grid $\cG_{n,s}$
with mesh-sizes
$$
m_j=2\,\left\lceil\,
    \left(\frac{\log\,M}{a_j^{\beta}\log\,\omega^{-1}}\right)^{1/b_j}\right
\rceil\,-\,1\qquad\mbox{for all}\qquad j=1,2,\dots,s.
$$
Here $M>1$ and $\beta\in(0,1)$. 
Note that $m_j\ge1$ and is always an odd number. Furthermore $m_j=1$ 
if $a_j\ge ((\log M)/(\log \omega^{-1}))^{1/\beta}$. 
Assume that $\alpha^* \in (0,\infty]$.
Since for all $\delta\in(0,\alpha^*)$ we have 
$$
a_j\ge \exp(\delta j) \quad\mbox{for all}\quad j\ge j^*_\delta,
$$ 
see~\eqref{newexpoaj}, we conclude that 
$$
j\ge j^{*}_{\beta,\delta}:=\max\left(j^*_\delta,
\frac{\log(((\log M)/(\log\omega^{-1}))^{1/\beta})}
{\delta}\right)
\ \ \ \mbox{implies}\ \ \ m_j=1.
$$ 
{}From~\eqref{gen_approx_bd} we have
$$
e_{n,s}^2:=
[e^{L_2-{\rm app}}(H(K_{s,\bsa,\bsb}),A_{n,s,M})]^2\le \frac1M\,+\,
\sum_{\bsh\in\mathcal{A}(s,M)}\
\sum_{\bsl\in\cG_{n,s}^\perp\setminus\{\bszero\}}\omega_{\bsh+\bsl}.
$$
We now estimate 
$$
\sum_{\bsl \in \mathcal{G}^\perp_{n,s} \setminus \{\bszero\} } 
\omega_{\bsh+\bsl}=  
\sum_{\emptyset \neq \uu \subseteq\{1,\ldots, s\}} 
\prod_{j\in \uu}\left( 
\sum_{\ell_j \in \mathbb{Z}\setminus \{0\}}\omega^{a_j|h_j+m_j \ell_j|^{b_j}}
\right)\, \prod_{j \not \in \uu}
\omega^{a_j|h_j|^{b_j}},
$$
where we separated the cases for $\ell_j \in \ZZ\setminus \{0\}$
and $\ell_j=0$. We estimate the second
product by one so that  
$$
\sum_{\bsl \in \mathcal{G}^\perp_{n,s} \setminus \{\bszero\} } 
\omega_{\bsh+\bsl}\le 
\sum_{\emptyset \neq \uu \subseteq\{1,\ldots, s\}} 
\prod_{j\in \uu}\left( 
\sum_{\ell \in \mathbb{Z}\setminus \{0\}}\omega^{a_j|h_j+m_j \ell|^{b_j}}
\right).
$$
We now show that for $\bsh\in \mathcal{A}(s,M)$ we have $|h_j|<(m_j+1)/2$
for all $j=1,2,\dots,s$. Indeed, the vector~$\bsh$ satisfies
$\prod_{j=1}^s\omega^{-a_j|h_j|^{b_j}}<M$, and since each factor is at
least one we have 
$\omega^{-a_j|h_j|^{b_j}}<M$ for all $j$, which is equivalent to
$$
|h_j|<
\left(\frac{\log\,M}{a_j\,\log\,\omega^{-1}}\right)^{1/b_j}\le
\left(\frac{\log\,M}{a_j^{\beta}\,\log\,\omega^{-1}}\right)^{1/b_j}
\le
\frac{m_j+1}2,
$$
as claimed.  

In particular, if $m_j=1$ then $h_j=0$ and 
\begin{equation}\label{ineqhone}
\sum_{\ell\in \mathbb{Z}\setminus \{0\}}\omega^{a_j|h_j+m_j
  \ell|^{b_j}}=2\sum_{\ell=1}^\infty\omega^{a_j\ell^{b_j}}
\le
2\sum_{\ell=1}^\infty\omega^{a_j\ell}=\frac{2\,\omega^{a_j}}{1-\omega^{a_j}}
\le\frac{2\,\omega^{a_j}}{1-\omega}.
\end{equation}

Let $m_j\ge3$. Then  
$|h_j|<(m_j+1)/2$. Since both $|h_j|$ and $(m_j+1)/2$ are positive
integers, we conclude that $|h_j|\le (m_j+1)/2-1=(m_j-1)/2$ and
$\ell\not=0$ implies 
$$
|h_j+m_j\ell|\ge m_j|\ell|-|h_j|\ge
\frac{m_j+1}{2} |\ell|.
$$
Therefore 
\begin{equation}\label{ineqhtwo}
\sum_{\ell \in \mathbb{Z}\setminus \{0\}}\omega^{a_j|h_j+m_j
  \ell|^{b_j}}\le
2\sum_{\ell=1}^\infty\omega^{a_j[(m_j+1)/2]^{b_j}\ell^{b_j}}
\le\frac{2\,\omega^{a_j[(m_j+1)/2]^{b_j}}}{1-\omega}.
\end{equation}
The inequalities \eqref{ineqhone} and \eqref{ineqhtwo} can be combined as
$$
\beta_j:=\sum_{\ell \in \mathbb{Z}\setminus \{0\}}\omega^{a_j|h_j+m_j
  \ell|^{b_j}}\le
\frac{2\,\omega^{a_j[(m_j+1)/2]^{b_j}}}{1-\omega}.
$$
Note that 
$$
\sum_{\emptyset \neq \uu \subseteq\{1,\ldots, s\}} 
\prod_{j\in \uu}\left(\sum_{\ell \in \mathbb{Z}\setminus \{0\}} 
\omega^{a_j |h_j+ m_j \ell|^{b_j}}\right)
=-1+
\sum_{\uu \subseteq\{1,\ldots, s\}} 
\prod_{j\in \uu} \beta_j=-1+\prod_{j=1}^s(1+\beta_j).
$$
Consequently,
$$
e_{n,s}^2\le \frac{1}{M}+ |\mathcal{A}(s,M)| 
\left(-1 + \prod_{j=1}^s 
\left(1 + \frac{2\,\omega^{\,a_j[(m_j+1)/2]^{b_j}}}{1-\omega} 
\right)\right).
$$
Using $\log (1+x) \le x$ we obtain
$$
\log\left[ \prod_{j=1}^s 
\left(1 + \frac{2\,\omega^{\,a_j[(m_j+1)/2]^{b_j}}}{1-\omega} 
\right)\right] \le  
\frac{2}{1-\omega} \sum_{j=1}^s\omega^{\,a_j[(m_j+1)/2]^{b_j}}.
$$
{}From the definition of $m_j$ we have
$
a_j[(m_j+1)/2]^{b_j}\ge a_j^{1-\beta}\,(\log\,M)/\log\,\omega^{-1}$.
Therefore
$$
\omega^{a_j[(m_j+1)/2]^{b_j}}\le
\omega^{a_j^{1-\beta}\,(\log\,M)/\log\,\omega^{-1}}=
\left(\frac1M\right)^{a_j^{1-\beta}}.
$$
Since $a_j\ge1$ for $j\le j^*_{\beta,\delta}-1$ and
$a_j\ge \exp(\delta j)$ for $j\ge j^*_{\beta,\delta}$ 
we obtain
$$
\gamma:=\frac{2}{1-\omega} 
\sum_{j=1}^s\omega^{\,a_j[(m_j+1)/2)]^{b_j}}\le
\frac{2}{1-\omega}\,\left(
\frac{j^*_{\beta,\delta}-1}M+\sum_{j=j^*_{\beta,\delta}}^\infty\left(\frac1M
\right)^{\exp((1-\beta) \delta j)}\right)
=\frac{C_{\beta,\delta}}M,
$$
where
$$
C_{\beta,\delta}
:=\frac{2}{1-\omega}\,\left(j^*_{\beta,\delta}
-1+\sum_{j=j^*_{\beta,\delta}}^\infty
\left(\frac1M\right)^{\exp((1-\beta) \delta j)-1}\right)<\infty.
$$
Note that for $M\ge C_{\beta,\delta}$ we have $\gamma\le 1$.

Using convexity we easily check that
$-1+\exp(\gamma) \le (\mathrm{e}-1)\gamma$ for all $\gamma\in[0,1]$. Thus
for $M\ge C_{\beta,\delta}$ we obtain
\begin{eqnarray*}
-1 + \prod_{j=1}^s 
\left(1 + 
\frac{2\,\omega^{\,a_j[(m_j+1)/2]^{b_j}}}{1-\omega}\right) 
& \le & -1 + \exp\left(\frac{2}{1-\omega}   
\sum_{j=1}^s \omega^{\,a_j[(m_j+1)/2]^{b_j}}\right)\\
&=&-1+\exp(\gamma)\le (\mathrm{e}-1)\gamma\\ 
& \le & \frac{C_{\beta,\delta}\,(\mathrm{e}-1)}M.
\end{eqnarray*}
We now turn to $|\mathcal{A}(s,M)|$
which was already estimated in the proof for the class
$\Lambda^{\rm all}$, see \eqref{bdnxs}. We have 
$$
|\mathcal{A}(s,M)|\le3^{j^*_{\beta,\delta}}
\left(1+\frac{\log\,M}{\log\,\omega^{-1}}
\right)^{B+(\log 3)/\delta}.
$$
Therefore
$$
e^2_{n,s}\le \frac1M\,\left[1+C_{\beta,\delta}
(\mathrm{e}-1)3^{j^*_{\beta,\delta}}
\left(1+\frac{\log\,M}{\log\,\omega^{-1}}\right)^{B+
(\log 3)/\delta}
\right]\le \frac{D_{\beta,\delta}}{\sqrt{M}},
$$
where 
$$
D_{\beta,\delta}:=\sup_{x\ge C_{\beta,\delta}}\left(
\frac1{\sqrt{x}}+\frac{C_{\beta,\delta}(\mathrm{e}-1)3^{j^*}}{\sqrt{x}}\,
\left(1+\frac{\log\,x}{\log\,\omega^{-1}}\right)^{B+
(\log 3)/\delta}\right)<\infty.
$$
Hence for
$$
M=\max(C_{\beta,\delta},D_{\beta,\delta}^2\,\e^{-4})
$$
we have 
$$
e_{n,s}=e^{L_2-{\rm app}}(H(K_{s,\bsa,\bsb}),A_{n,s,M})\le \e.
$$
We estimate the number $n$ of function values used by the 
algorithm $A_{n,s,M}$. We have 
\begin{eqnarray*}
n&=&\prod_{j=1}^sm_j=\prod_{j=1}^{\min(s,j^*_{\beta,\delta})}m_j\le
\prod_{j=1}^{\min(s,j^*_{\beta,\delta})}
\left(1+2\left(\frac{\log\,M}{a_j^{\beta}\,
\log\,\omega^{-1}}\right)^{1/b_j}\right)\\
&\le& 3^{j^*_{\beta,\delta}}\,
\left(\frac{\log\,M}{\log\,\omega^{-1}}\right)^B
\le 
3^{j^*_{\beta,\delta}}
\,\left(\frac{\log\,M}{\log\,\omega^{-1}}
\right)^{B+(\log 3)/(\beta\,\delta)}\\
&=&\mathcal{O}\left(\left(1+\log\,\e^{-1}
\right)^{B+(\log 3)/(\beta\,\delta)}\right),
\end{eqnarray*}
where the factor in the big $\mathcal{O}$ notation depends only on
$\beta$ and $\delta$. This proves SPT+UEXP with
$$
\tau=B+\frac{\log 3}{\beta\,\delta}.
$$
Since $\beta$ can be arbitrarily close to one, and $\delta$ can be
arbitrarily close to $\alpha^*$,
the exponent $\tau^*$ of SPT is at most
$$
B+\frac{\log 3}{\alpha^*},
$$
where for $\alpha^*=\infty$ we have $\frac{\log 3}{\alpha^*} = 0$.
This completes the proof of Theorem~\ref{mainresult} for the
class~$\Lambda^{\rm std}$. 
\section{Relations to multivariate integration}\label{secint}
Multivariate integration 
$$
{\rm INT}_s(f)=\int_{[0,1]^s}f(\bsx)\,{\rm d}\bsx
$$
for $f$ from the Korobov space $H(K_{s,\bsa,\bsb})$ was studied 
in~\cite{KPW12}. It is easy to see that 
multivariate approximation is not easier than 
multivariate integration, see e.g., \cite{NSW04}. 
More precisely, 
for any algorithm $A_{n,s}(f)=\sum_{k=1}^n \alpha_k f(\bsx_k)$ 
for multivariate approximation 
using the nodes $\bsx_1,\ldots,\bsx_n\in [0,1)^s$ and $\alpha_k\in 
L_2([0,1]^s)$, define 
$\beta_k:=\int_{[0,1]^s} \alpha_k (\bsx)\rd \bsx$ and the algorithm
$$
A^{\rm int}_{n,s}(f)=\sum_{k=1}^n\beta_k\,f(\bsx_k)
$$
for multivariate integration. Then
\begin{eqnarray*}
\abs{\int_{[0,1]^s}f(\bsx)\rd \bsx - A^{\rm int}_{n,s}(f)}&=&
\left|\int_{[0,1]^s}\left(f(\bsx)-\sum_{k=1}^n\alpha_k(\bsx)\,f(\bsx_k)\right)
\,{\rm d}\bsx\right|\\
&\le&\left(\int_{[0,1]^s}
\left(f(\bsx)-\sum_{k=1}^n\alpha_k(\bsx)\,f(\bsx_k)\right)^2
\,{\rm d}\bsx\right)^{1/2}\\
&=&\|f-A_{n,s}(f)\|_{L_2([0,1]^s)}.
\end{eqnarray*}
This proves that for the worst-case error for integration we have 
$$
e(H(K_{s,\bsa,\bsb}),A^{\rm int}_{n,s}):=
\sup_{\substack{f\in H(K_{s,\bsa,\bsb})\\ \norm{f}_{H(K_{s,\bsa,\bsb})}\le 1}}
\abs{\int_{[0,1]^s}f(\bsx)\rd \bsx - \sum_{k=1}^n \beta_k f(\bsx_k)}\le
e^{L_2-{\rm app}}(H(K_{s,\bsa,\bsb}),A_{n,s}).
$$
Since this holds for all algorithms $A_{n,s}$ we conclude that
\begin{equation}\label{eqintapp}
e^{{\rm int}}(n,s):=\inf_{A^{\rm int}_{n,s}}
e(H(K_{s,\bsa,\bsb}),A^{\rm int}_{n,s})\le
e^{\rm app}(n,s):=e^{L_2-\mathrm{app},\Lambda^{\rm{std}}}(n,s).
\end{equation}
Here  $e^{{\rm int}}(n,s)$ and $e^{{\rm app}}(n,s)$ are 
the $n$th minimal worst-case errors for multivariate 
integration and approximation in $H(K_{s,\bsa,\bsb})$, respectively. 
Furthermore for $n=0$ we have equality,
$$
e^{{\rm int}}(0,s)=e^{{\rm app}}(0,s)=1.
$$ 
{}From these observations it follows that 
for $\varepsilon \in (0,1)$ and $s \in \NN$ we have 
\begin{equation}\label{eqintapp2}
n^{{\rm int}}(\varepsilon,s) 
\le n^{L_2-\mathrm{app},\Lambda^{\rm{std}}}(\varepsilon,s),
\end{equation}
where $n^{{\rm int}}(\varepsilon,s)$ is the 
information complexity for the integration problem.

Obviously, for multivariate integration only the class $\Lambda^{\rm
  std}$ makes sense. 
The inequalities \eqref{eqintapp} and \eqref{eqintapp2} mean that all positive
results for multivariate approximation and the class~$\Lambda^{\rm
std}$ also hold for multivariate integration. In particular,
the following facts hold:
\begin{itemize}
\item
Exponential convergence holds for multivariate integration 
for arbitrary  $\bsa$ and $\bsb$ with the largest rate $p^{\rm int}(s)\ge
1/B(s)$. Although only uniform exponential convergence was considered 
in~\cite{KPW12}, the proof presented there allows to conclude that
we have $p^{\rm int}(s)=1/B(s)$.
\item Uniform convergence holds for multivariate integration 
iff $B<\infty$ and 
the largest rate $[p^{\rm int}]^*=1/B$, as for multivariate
  approximation. This was shown  in~\cite{KPW12}. 
\item Polynomial tractability and strong polynomial tractability for
  multivariate integration were studied in~\cite{KPW12}, where it was
  shown that they are
  equivalent and hold iff $B<\infty$ and $a_j$'s are exponentially
  growing with $j$. These conditions are the same as for
  multivariate approximation.
 
  The exponent $[\tau^{{\rm int}}]^{\ast}$ of SPT for 
multivariate integration was estimated in 
\cite{KPW12}, and we have $[\tau^{{\rm int}}]^{\ast} \in [B,2B]$. 
{}From Theorem~\ref{mainresult} it follows that 
$$[\tau^{{\rm int}}]^{\ast} 
\le [\tau^{{\rm app}}]^{\ast} \le B+ \frac{\log 3}{\alpha^*},$$ 
where $[\tau^{{\rm app}}]^{\ast}$ 
is the exponent of SPT for the approximation problem. 
Hence we have 
$$[\tau^{{\rm int}}]^{\ast} 
\in \left[B,B+\min\left(B,\frac{\log 3}{\alpha^*}\right)\right],
$$ 
which is an improvement of the result from \cite{KPW12} whenever 
$\alpha^* > (\log3)/B$ 
which means that $a_j > \exp(j(\alpha^*-\delta))$ for large  $j$.
If $\alpha^*=\infty$ then 
$$
[\tau^{{\rm int}}]^{\ast}= [\tau^{{\rm app}}]^{\ast}=B.
$$
This is the case when $a_j\ge(1+\alpha)^{b_j}$ for large $j$ and
$\alpha>0$.

\item 
Weak tractability  for the integration problem 
was considered in \cite{KPW12} 
with a more demanding notion of WT. Suppose that we relax the notion of WT 
from~\cite{KPW12}, and use the notion of WT studied in 
this paper. That is, we say that the integration problem is weakly
tractable if
\begin{equation}\label{WTint}
\lim_{s+\log \varepsilon^{-1} \rightarrow 
\infty}\frac{\log n^{{\rm int}}(\varepsilon,s)}{s+\log \varepsilon^{-1}}=0.
\end{equation}
We stress that the notion of WT as discussed in \cite{KPW12} 
implies \eqref{WTint}, but this does not hold the other way round. 

Using the definition~\eqref{WTint}, 
we now show that we have the same condition $\lim_ja_j=\infty$
for WT for the integration and 
approximation problems. Indeed, 
by Theorem~\ref{mainresult}, the condition $\lim_j a_j= \infty$
implies WT for the approximation problem, which, by \eqref{eqintapp2}, 
also implies WT for the integration problem. 
To show the converse, assume that the $a_j$'s 
are bounded, say $a_j\le A<\infty$ 
for all $j \in \NN$. {}From \cite[Corollary~1]{KPW12} 
it follows that for all $n < 2^s$ we have 
$$
e^{{\rm int}}(n,s) 
\ge 2^{-s/2} \,\omega^{2^{-1} \sum_{j=1}^s a_j}
\ge 2^{-s/2}\,\omega^{A s/2}= \eta^s,
$$ 
where $\eta:= (\omega^{A}/2)^{1/2} \in (0,1)$. 
Hence, for $\varepsilon=\eta^s/2$ we have 
$e^{{\rm int}}(n,s)>\e$ for all $n<2^s$. This implies that 
$n^{{\rm int}}(\varepsilon,s)\ge 2^s$ and  
$$\frac{\log n^{{\rm int}}(\varepsilon,s)
}{s+\log \varepsilon^{-1}} 
\ge \frac{s \log 2}{s+\log 2 +s \log \eta^{-1}} 
\rightarrow  \frac{\log 2}{1+\log \eta^{-1}} >0
\quad \mbox{as\ \ $s \rightarrow \infty$}. 
$$ 
Thus we do not have WT.

This means that WT holds in the sense of \eqref{WTint} 
for the integration problem  iff $\lim_j a_j=\infty$, 
which is the same condition as for the approximation problem. 

Since for the integration problem we have UEXP iff $B< \infty$, 
see \cite[Theorem~1]{KPW12}, 
it follows that we have WT+UEXP iff $B< \infty$ and $\lim_j a_j=\infty$. 
Again, this is the same condition as for the approximation problem.
\end{itemize}

\section{Analyticity of functions from $H(K_{s,\bsa,\bsb})$}
\label{analyticfunctions}
In this section we show that 
the functions from the Korobov space $H(K_{s,\bsa,\bsb})$ are 
analytic.

\begin{proposition}
Functions $f \in H(K_{s,\bsa,\bsb})$ are analytic.
\end{proposition}

\begin{proof}
Since $H(K_{s,\bsa,\bsb}) \subseteq H(K_{s,\bsone,\bsone})$ 
it suffices to show the assertion for $f\in H(K_{s,\bsone,\bsone})$.

Let
$\bsalpha=(\alpha_1,\alpha_2,\dots,\alpha_s) \in \NN_0^s$ with
$|\bsalpha|=\alpha_1+\cdots +\alpha_s$. 
For $f\in H(K_{s,\bsone,\bsone})$, 
consider the operator $D^{\bsalpha}$ of partial differentiation,
$$
D^{\bsalpha} f=\frac{\partial^{\,|\bsalpha|}}{\partial x_1^{\alpha_1}\
\partial x_2^{\alpha_2}\ \cdots\ \partial x_s^{\alpha_s}}f.
$$
Then 
$$
D^{\bsalpha} f(\bsx)=\sum_{\bsh\in\ZZ^s}\left[
\widehat f(\bsh)\,(2\pi\icomp)^{|\bsalpha|}
\prod_{j=1}^sh_j^{\alpha_j}\right]\, 
\exp(2\pi \icomp \bsh \cdot\bsx),
$$
where, by convention, we take $0^0=1$. 

Let 
$\omega_1 \in (\omega,1)$ and 
$q=\omega/\omega_1 < 1$. 
For any $\alpha \in \NN$ consider
$g(x) =x^{2 \alpha} q^x$ for $x \ge 0$. 
Then $g'(x)=0$ if $x=2 \alpha/\log q^{-1}$ and 
$$
g''\left( 2 \alpha/\log q^{-1}\right)=
\frac{1}{2} \left(\frac{2}{{\rm e}}\right)^{2\alpha} 
\left(\frac{\alpha}{\log q^{-1}}\right)^{2\alpha -1} \log q < 0.
$$ 
Hence,
$$
g(x) \le g\left( 2 \alpha/\log q^{-1}\right)= 
\left(\frac{2 \alpha}{{\rm e}\,\log q^{-1}}\right)^{2 \alpha}.
$$ 
Since
$$
\alpha^{2 \alpha} = 
\left(\alpha! \frac{\alpha^{\alpha}}{\alpha!}\right)^2 
\le {\rm e}^{2 \alpha} (\alpha!)^2
$$ 
then
$$
g(x) \le \left(\frac{2}{ \log q^{-1}}\right)^{2 \alpha}  (\alpha!)^2.
$$ 
Hence, we have 
$$
x^{2 \alpha} \omega^x \le 
\left(\frac{2}{ \log \omega_1 - \log \omega}\right)^{2 \alpha}  
(\alpha!)^2 \omega_1^x=:C^{2 \alpha} (\alpha!)^2 \omega_1^x.
$$ 
Note that $C$ depends only on $\omega$ and $\omega_1$.

Then $\omega_{\bsh}=\omega^{|h_1|+\cdots+|h_s|}$ implies
\begin{eqnarray*}
|D^{\bsalpha} f(\bsx)|& = &\left|\sum_{\bsh\in\ZZ^s}\left[
\widehat f(\bsh)\omega_{\bsh}^{-1/2}\right] 
\left[\omega_{\bsh}^{1/2}(2\pi\icomp)^{|\bsalpha|}
\prod_{j=1}^s h_j^{\alpha_j}\right]
\exp(2\pi \icomp \bsh \cdot\bsx)\right|\\
& \le &  \|f\|_{H(K_{s,\bsone,\bsone})}\,
\left[\sum_{\bsh\in\ZZ^s}(2\pi)^{2|\bsalpha|}
\prod_{j=1}^s|h_j|^{2\alpha_j}\omega^{|h_j|}\right]^{1/2}\\
&\le &  \|f\|_{H(K_{s,\bsone,\bsone})}\,
\left[\sum_{\bsh\in\ZZ^s}(2\pi)^{2|\bsalpha|}
\prod_{j=1}^s \left[C^{2\alpha_j} (\alpha_j!)^2\right] 
\omega_1^{|h_j|}\right]^{1/2}\\
&\le &  \|f\|_{H(K_{s,\bsone,\bsone})}\, 
\prod_{j=1}^s \left[(2 \pi C)^{\alpha_j} 
\alpha_j!\right] \left[\sum_{\bsh\in\ZZ^s}
\prod_{j=1}^s \omega_1^{|h_j|}\right]^{1/2}\\
&\le & \|f\|_{H(K_{s,\bsone,\bsone})}\, (2 \pi C)^{|\bsalpha|} 
\prod_{j=1}^s \left(\alpha_j!\right) \left(1+\frac2{1-\omega_1}\right)^{s/2}\\
& =: & C_1 \cdot C_2^{|\bsalpha|} \prod_{j=1}^s \left(\alpha_j!\right),
\end{eqnarray*}
where $C_1=\|f\|_{H(K_{s,\bsone,\bsone})} 
\left(1+\frac2{1-\omega_1}\right)^{s/2} \ge 0$ and $C_2=2 \pi C >0$.

Then for any $\bszeta =(\zeta_1,\ldots, \zeta_s)$ and any  
$\bsx=(x_1,\ldots,x_s)$ with $\|\bsx-\bszeta\|_{\infty} < C_2^{-1}$ we have
\begin{eqnarray*}
\left|\sum_{\bsalpha \in \NN_0^s} 
\frac{D^{\bsalpha} f(\bszeta)}{(\alpha_1!) \cdots (\alpha_s!)}\, 
\prod_{j=1}^s (x_j-\zeta_j)^{\alpha_j}\right|  & \le & 
C_1 \sum_{\bsalpha \in \NN_0^s}\, 
\prod_{j=1}^s (C_2|x_j-\zeta_j|)^{\alpha_j}\\
& \le & C_1\left(\sum_{\alpha=0}^{\infty} 
(C_2 \| \bsx-\bszeta\|_{\infty})^{\alpha}\right)^s\\
&=&C_1\left(\frac1{1-C_2\|\bsx-\bszeta\|_{\infty}}\right)^s < \infty.
\end{eqnarray*}
Hence $f$ is analytic, as claimed.
\end{proof}

\begin{small}
\noindent\textbf{Authors' addresses:}
\\ \\
\noindent Josef Dick,
\\ 
School of Mathematics and Statistics, 
University of New South Wales, Sydney, NSW, 2052, Australia\\
 \\
\noindent Peter Kritzer, Friedrich Pillichshammer,
\\ 
Institut f\"{u}r Finanzmathematik, 
Universit\"{a}t Linz, Altenbergerstr.~69, 4040 Linz, Austria\\
 \\
\noindent Henryk Wo\'{z}niakowski, \\
Department of Computer Science, Columbia University, New York 10027,
USA and Institute of Applied Mathematics, 
University of Warsaw, ul. Banacha 2, 02-097 Warszawa, Poland\\ \\

\noindent \textbf{E-mail:} \\
\texttt{josef.dick@unsw.edu.au}\\
\texttt{peter.kritzer@jku.at}\\
\texttt{friedrich.pillichshammer@jku.at} \\
\texttt{henryk@cs.columbia.edu}
\end{small}

\begin{thebibliography}{99}

\bibitem{Aron} N. Aronszajn, Theory of reproducing kernels. 
Trans. Amer. Math. Soc. 68, 337--404, 1950


\bibitem{DLPW11} J.~Dick, G.~Larcher, F.~Pillichshammer,
H.~Wo\'{z}niakowski. Exponential convergence and tractability of
multivariate integration for Korobov spaces. Math. Comp. 80,
905--930, 2011.



\bibitem{GW09} M. Gnewuch, H. Wo\'zniakowski,
Generalized tractability for multivariate problems,
Part II: Linear tensor product problems, linear information,
unrestricted tractability.
Found. Comput. Math. 9, 431--460, 2009.


\bibitem{KPW12} P. Kritzer, F. Pillichshammer, H. Wo\'{z}niakowski. 
Multivariate integration
of infinitely many times differentiable functions 
in weighted Korobov spaces. To appear in Math. Comp., 2013.

\bibitem{KSW06} F.Y. Kuo, I.H. Sloan, H. Wo\'{z}niakowski. Lattice
rules for multivariate approximation in the worst case setting.
In: H. Niederreiter, D. Talay (eds.). \textit{Monte Carlo and
Quasi-Monte Carlo Methods 2004}. Springer, Berlin, pp. 289--330, 2006.

\bibitem{NSW04} E.~Novak, I.H.~Sloan, H.~Wo\'{z}niakowski. Tractability of
approximation for weighed Korobov spaces on 
classical and quantum computers. Found. Comput. Math. 4, 121--156, 2004.

\bibitem{NW08}
E.~Novak and H.~Wo\'zniakowski. \textit{Tractability of
Multivariate Problems, Volume I: Linear
  Information}. EMS, Zurich, 2008.

\bibitem{NW10}
E.~Novak and H.~Wo\'zniakowski. \textit{Tractability of
Multivariate Problems, Volume II: Standard Informations for
Functionals}. EMS, Zurich, 2010.

\bibitem{NW12}
E.~Novak and H.~Wo\'zniakowski. \textit{Tractability of
Multivariate Problems, Volume III: Standard Informations for
Operators}. EMS, Zurich, 2012.

\bibitem{TWW88}
J.F.~Traub, G.W.~Wasilkowski, and H.~Wo\'{z}niakowski. \newblock
{\em Information-Based Complexity}. \newblock Academic Press, New
York, 1988.

\end{thebibliography}
\end{document}